\documentclass[10pt,reqno]{amsart} 

\usepackage{amsthm,amsfonts,amscd,amsmath}
\usepackage[hidelinks]{hyperref} 
\usepackage{url}
\usepackage{xcolor}
\hypersetup{
    colorlinks,
    linkcolor={red!50!black},
    citecolor={blue!50!black},
    urlcolor={blue!80!black}
}

\usepackage[normalem]{ulem}
\usepackage{graphicx} 
\usepackage{datetime} 
\usepackage{arydshln} 
\usepackage{cancel}
\usepackage{caption,subcaption}
\numberwithin{figure}{section}
\numberwithin{table}{section}

\newcommand{\cf}{\textit{cf.\ }} 
\newcommand{\Iverson}[1]{\ensuremath{\left[#1\right]_{\delta}}} 

\DeclareMathOperator{\DGF}{DGF} 
\DeclareMathOperator{\ds}{ds} 
\DeclareMathOperator{\Id}{Id}

\title[Generalized Lambert Series Factorizations]{
       Factorization Theorems for Generalized Lambert Series and Applications
} 
\author[Mircea Merca and Maxie D. Schmidt]{
        Mircea Merca \\ 
        Academy of Romanian Scientists \\ 
        Splaiul Independentei 54, Bucharest, 050094 Romania \\ 
        \href{mailto:mircea.merca@profinfo.edu.ro}{mircea.merca@profinfo.edu.ro} \\ 
        \\ 
        Maxie D. Schmidt \\ 
        School of Mathematics \\ 
        Georgia Institute of Technology \\ 
        Atlanta, GA 30332 USA \\ 
        \href{mailto:maxieds@gmail.com}{maxieds@gmail.com} \\ 
        \href{mailto:mschmidt34@gatech.edu}{mschmidt34@gatech.edu}
} 
\date{\today} 

\allowdisplaybreaks 

\theoremstyle{plain} 
\newtheorem{theorem}{Theorem}

\newtheorem{prop}[theorem]{Proposition}

\newtheorem{cor}[theorem]{Corollary}
\numberwithin{theorem}{section}

\theoremstyle{definition} 
\newtheorem{example}[theorem]{Example}
\newtheorem{remark}[theorem]{Remark}

\begin{document} 

\begin{abstract} 
We prove new variants of the Lambert series factorization theorems studied by 
Merca and Schmidt (2017) which correspond to a more general class of 
Lambert series expansions of the form 
$L_a(\alpha, \beta; q) := \sum_{n \geq 1} a_n q^{\alpha n-\beta} / (1-q^{\alpha n-\beta})$ 
for integers $\alpha, \beta$ defined such that $\alpha \geq 1$ and $0 \leq \beta < \alpha$. 
Applications of the new results in the article are given to restricted divisor sums over 
several classical special arithmetic functions which define the cases of well-known, so-termed 
``ordinary'' Lambert series expansions cited in the introduction. 
We prove several new forms of factorization theorems for Lambert series 
over a convolution of two arithmetic functions which similarly lead to new 
applications relating convolutions of special multiplicative functions to partition 
functions and $n$-fold convolutions of one of the special functions. 
\end{abstract}

\maketitle

\section{Introduction} 

\subsection{Factorizations of generalized Lambert series} 

For fixed $\alpha, \beta \in \mathbb{Z}$ such that $\alpha \geq 1$ and $0 \leq \beta < \alpha$, and an 
arbitrary sequence $\{a_n\}_{n \geq 1}$, 
we consider generalized Lambert series expansions of the form 
\begin{align}
\label{eqn_GenLambertSeries_LaAlphaBetaq_def}
L_a(\alpha, \beta; q) & := \sum_{n \geq 1} \frac{a_n q^{\alpha n-\beta}}{1-q^{\alpha n-\beta}} = 
     \sum_{m \geq 1} b_m \cdot q^m,\ 
     |q^{\alpha}| < 1. 
\end{align} 
The coefficients of the generalized Lambert series expansion on the left-hand-side of the previous 
equation are given by
\begin{align*} 
b_m & = \sum_{\substack{\alpha d-\beta | m}} a_d. 
\end{align*} 
Several well known variants of the ordinary Lambert series expansions studied in 
\cite{MERCA-SCHMIDT1,MERCA-SCHMIDT2,MERCA-LSFACTTHM,SCHMIDT-LSFACTTHM} 
which generate special arithmetic functions are transformed into a 
series of the form in \eqref{eqn_GenLambertSeries_LaAlphaBetaq_def} as follows 
where $\mu(n)$ denotes the 
\emph{M\"obius function}, $\phi(n)$ denotes \emph{Euler's totient function}, 
$\sigma_{\alpha}(n)$ denotes the generalized \emph{sum of divisors function} for 
a fixed $\alpha \in \mathbb{C}$, 
$\lambda(n)$ denotes \emph{Liouville's function}, $\Lambda(n)$ denotes 
\emph{von Mangoldt's function}, $\omega(n)$ defines the number of 
distinct primes dividing $n$, and $J_t(n)$ is 
\emph{Jordan's totient function} for some fixed $t \in \mathbb{C}$
\cite[\S 27]{NISTHB} \cite[\cf \S 1, \S 3]{MERCA-SCHMIDT1}:
\begin{align} 
\label{eqn_WellKnown_LamberSeries_Examples} 
\sum_{n \geq 1} \frac{\mu(n) q^{\alpha n - \beta}}{1-q^{\alpha n - \beta}} & = 
     \sum_{m \geq 1} \sum_{\alpha d-\beta|m} \mu(d) \cdot q^m \\ 
\notag
\sum_{n \geq 1} \frac{\phi(n) q^{\alpha n - \beta}}{1-q^{\alpha n - \beta}} & = 
     \sum_{m \geq 1} \sum_{\alpha d-\beta|m} \phi(d) \cdot q^m  \\ 
\notag
\sum_{n \geq 1} \frac{n^{x} q^{\alpha n - \beta}}{1-q^{\alpha n - \beta}} & =  
     \sum_{m \geq 1} \sum_{\alpha d-\beta|m} d^x \cdot q^m  \\ 
\notag
\sum_{n \geq 1} \frac{\lambda(n) q^{\alpha n - \beta}}{1-q^{\alpha n - \beta}} & = 
     \sum_{m \geq 1} \sum_{\alpha d-\beta|m} \lambda(d) \cdot q^m \\ 
\notag 
\sum_{n \geq 1} \frac{\Lambda(n) q^{\alpha n - \beta}}{1-q^{\alpha n - \beta}} & = 
     \sum_{m \geq 1} \sum_{\alpha d-\beta|m} \Lambda(d) \cdot q^m \\ 
\notag 
\sum_{n \geq 1} \frac{|\mu(n)| q^{\alpha n - \beta}}{1-q^{\alpha n - \beta}} & = 
     \sum_{m \geq 1} \sum_{\alpha d-\beta|m} |\mu(d)| \cdot q^m \\ 
\notag 
\sum_{n \geq 1} \frac{J_t(n) q^{\alpha n - \beta}}{1-q^{\alpha n - \beta}} & = 
     \sum_{m \geq 1} \sum_{\alpha d-\beta|m} J_t(d) \cdot q^m. 
\end{align}
Moreover, in the special case where $(\alpha, \beta) := (2, 1)$, we have another Lambert series expansion 
generating the \emph{sum of squares function}, $r_2(n)$, of the form \cite[\S 17.10]{HARDYANDWRIGHT} 
\begin{align*} 
\sum_{m \geq 1} r_2(m) q^m & = \sum_{n \geq 1} \frac{4 \cdot (-1)^{n+1} q^{2n-1}}{1-q^{2n-1}}. 
\end{align*} 
For the remainder of the article we treat the generalized series parameters $\alpha, \beta$ to be 
defined by the constraints above and the sequence $\{a_n\}_{n \geq 1}$ to be arbitrary unless 
otherwise specified. 

Within this article, we extend the so-termed ``\textit{factorization theorems}'' proved in 
\cite{MERCA-SCHMIDT1,MERCA-SCHMIDT2,MERCA-LSFACTTHM,SCHMIDT-LSFACTTHM} 
to the generalized Lambert series cases 
defined in \eqref{eqn_GenLambertSeries_LaAlphaBetaq_def}. 
In particular, we consider factorizations of the form 
\begin{align} 
\label{eqn_GenFactThmExp_def_v1} 
L_a(\alpha, \beta; cq) & = \frac{1}{C(q)} \sum_{n \geq 1} \sum_{k=1}^n s_{n,k} \bar{a}_k (cq)^n, 
\end{align} 
where $\bar{a}_n$ depends only on the $s_{n,k}$ and on the sequence of $a_n$. 
In general, when $\alpha > 1$ and $\bar{a}_n \equiv a_n$ for all $n \geq 1$, 
we typically see that the square matrices, 
$A_n := (s_{i,j})_{1 \leq i,j \leq n}$, are not invertible as in the cases of the 
first factorization theorems proved in the references 
\cite{MERCA-SCHMIDT1,MERCA-SCHMIDT2,SCHMIDT-LSFACTTHM}. 
However, we may still proceed to define an inverse sequence, $s_{n,k}^{(-1)}$, as in the references 
with a corresponding non-singular matrix representation which implicitly defines the sequence of 
$s_{n,k}$ as in \cite{MERCA-SCHMIDT2} for $\bar{a}_n \not{\equiv} a_n$. In these cases, we have a 
matrix representation of the factorization theorem in \eqref{eqn_GenFactThmExp_def_v1} of the form 
\begin{align*} 
\begin{bmatrix} \bar{a}_1 \\ \bar{a}_2 \\ \vdots \\ \bar{a}_n \end{bmatrix} & = 
     A_n^{-1} \begin{bmatrix} B_0 \\ B_1 \\ \vdots \\ B_{n-1} \end{bmatrix}, 
\end{align*} 
where the sequence of $\{B_{m}\}_{m \geq 0}$ depends on the 
arithmetic function, $b_m$, implicit to the expansion of 
\eqref{eqn_GenLambertSeries_LaAlphaBetaq_def} and the factorization pair, $(C(q), s_{n,k})$. 
When the matrix $A_n$ is non-singular, we have the next determinant-based recurrence relations 
proved as in \cite[\S 2]{MERCA-SCHMIDT1} relating the two sequences, 
$s_{n,k}$ and $s_{n,k}^{(-1)}$. 
\begin{align} 
\label{eqn_snk_snkinv_det-based_recrels} 
s_{n,j}^{(-1)} & = - \sum_{k=1}^{n-j} s_{n,n+1-k}^{(-1)} \cdot s_{n+1-k,j} + 
     \delta_{n,j} \\ 
\notag 
     & = 
     - \sum_{k=1}^{n-j} s_{n,n-k} \cdot s_{n-k,j}^{(-1)} + \delta_{n,j} 
\end{align} 
These identities are symmetric in that these 
identities still hold if one sequence is interchanged with the other. 

\subsection{Significance of our new results} 

In this article, we prove a few variants corresponding to the expansions of 
\eqref{eqn_GenLambertSeries_LaAlphaBetaq_def} which effectively generalize the 
Lambert series factorization theorems found in the first references 
\cite{MERCA-SCHMIDT1,MERCA-SCHMIDT2}. 
Namely, we prove the key results in Theorem \ref{theorem1_GenFormula_for_snk} and 
Theorem \ref{theorem2_AnotherGenFactThm} and their corollaries 
which provide analogous factorization theorems for the generalized cases of the 
Lambert series defined in the first section of the introduction above. 
As in the similar and closely-related factorization results found in the references, 
each of these factorization theorems provide new relations between sums over an 
arbitrary function, $a_n$, involving the divisors of $n$ and more additive identities 
involving partition functions and the same function $a_n$. 
Thus the results proved in this article continue the spirit of 
\cite{MERCA-SCHMIDT1,MERCA-SCHMIDT2,MERCA-LSFACTTHM,SCHMIDT-LSFACTTHM} by 
connecting the at times seemingly disparate branches of additive and 
multiplicative number theory in new and interesting ways. 

Central to the importance of our new results are the applications to modified 
divisors sums for many classical special arithmetic functions often studied in additive and
multiplicative number theory. These special functions include the M\"obius function $\mu(n)$, 
Euler's totient function $\phi(n)$, the sum-of-divisors functions $\sigma_{\alpha}(n)$, 
Liouville's function $\lambda(n)$, von Mangoldt's function $\Lambda(n)$, $|\mu(n)|$, 
the number of distinct primes dividing $n$, $\omega(n)$, Jordan's totient function 
$J_t(n)$, and the sum of squares function $r_2(n)$. 
In particular, we provide a number of examples of our new results pertaining to 
these classical special functions in Section \ref{Section_GenFactThms_and_Examples} and 
Section \ref{Section_Variants_and_Examples}.  
Also discussed in this article are conjectures on the expansions of degenerate cases of 
Theorem \ref{theorem1_GenFormula_for_snk} given in terms of nested formulas 
involving Euler's partition function $p(n)$ (see Remark \ref{remark_conj_degen_cases_of_thm1}). 
These conjectured identities are interesting in form for their 
own sake, and fit in with the properties we establish for the generalized Lambert series 
cases in \eqref{eqn_GenLambertSeries_LaAlphaBetaq_def}. 

Finally, we conclude the article by stating and proving another somewhat more general 
class of factorization theorems for Lambert series over a convolution of arithmetic 
functions, $f \ast g$. These uniquely new results lead to still more applications 
connecting partition functions and special functions from multiplicative number theory. 
In particular, in Section \ref{Section_LSFactThms_CvlOfTwoFns_and_Apps} of the article 
we prove Proposition \ref{prop_OnePossibleLSFact} and 
Theorem \ref{claim_snk_inverses} 
providing the more difficult to obtain corresponding forms of the inverse sequences 
specifying these factorizations. 
Furthermore, the remarks given in Section \ref{subSection_Concl_OtherFactThmExps} 
offer several concluding suggestions on enumerating new, and more general 
factorization theorems for the expansions of Lambert series generating functions 
which we have so far not considered in this article or in the references. 
Our continued aim in exploring these new variations of the Lambert series 
factorization theorems is to branch out and provide further connections 
between the inherently multiplicative structure of the Lambert series 
generating functions and the additive theory of partitions and special 
partition functions. 

\section{Generalized Factorization Theorems} 
\label{Section_GenFactThms_and_Examples}

\subsection{Special Cases} 
\label{subSection_GenFactThmV1_SpCases_and_Examples} 

Before we state and prove the generalized factorization theorem results in the following  
subsections, we first consider the special case series expansions identified below. 
These next results given in this subsection identify several new interpretations of the 
factor pair sequence $s_{n,k}$ as well as provide generalized analogs to the 
expansions of several identities cited in \cite{MERCA-SCHMIDT2}. 

\begin{prop}
     \label{prop_first_spcase_result} 
	For $|q|<1$, we have that 
	\begin{align*}
	\sum_{n=1}^{\infty} a_n \frac{q^{2n-1}}{1-q^{2n-1}} = \frac{1}{(q;q^2)_\infty} 
	\sum_{n=1}^{\infty} \left( \sum_{k=1}^n s_{n,k} a_k\right) (-1)^{n-1} q^n,
	\end{align*}
	where $s_{n,k}$ denotes the number of $(2k-1)$'s in all partitions of $n$ 
	into distinct odd parts.
\end{prop} 
\begin{proof}
     We consider the identity \cite[eq. 2.1]{MERCA-LSFACTTHM}, namely
	\begin{align*}
	\sum_{k=1}^n \frac{a_k x_k}{1-x_k} = \left( \prod_{k=1}^{n} \frac{1}{1-x_k} \right) \left( 
	\sum_{k=1}^{n} \sum_{1\leq i_1 < \ldots <i_k\leq n} (-1)^{k-1} 
	(a_{i_1}+\cdots +a_{i_k}) x_{i_1}\cdots x_{i_k}\right) .
	\end{align*}
	By this relation with $x_k$ replaced by $q^{2k-1}$, we get
	\begin{align*}
     \sum_{k=1}^n & \frac{a_k q^{2k-1}}{1-q^{2k-1}} \\ 
     & = \frac{1}{(q;q^2)_n}  \left(  \sum_{k=1}^{n} \sum_{1\leq i_1 < \ldots <i_k\leq n} (-1)^{k-1} 
     (a_{i_1}+\cdots +a_{i_k}) q^{(2i_1-1)+\cdots+(2i_k-1)}\right) .
\end{align*}	
   The result follows directly from this identity considering the 
   limiting case as $n\to\infty$.
\end{proof}

\begin{example}[Consequences of the Proposition] 
The result in Proposition \ref{prop_first_spcase_result} 
allows us to derive many special case identities involving Euler's partition function and various arithmetic functions. 
More precisely, by the well-known famous special cases Lambert series identities 
expanded in the introduction to \cite{MERCA-SCHMIDT1}, for $n \geq 1$ we have that 
	\begin{align*}
	& \sum_{k=1}^{n} \sum_{2d-1|k} d^x \cdot \widetilde{q}(n-k) = \sum_{k=1}^n (-1)^{n-1} k^x s_{n,k},\\
	& \sum_{k=1}^{n} \sum_{2d-1|k} \mu(d) \cdot \widetilde{q}(n-k) = \sum_{k=1}^{n} (-1)^{n-1} \mu(k) s_{n,k}, \\
	& \sum_{k=1}^{n} \sum_{2d-1|k} \phi(d) \cdot \widetilde{q}(n-k) = \sum_{k=1}^n (-1)^{n-1} \phi(k) s_{n,k},\\
	& \sum_{k=1}^{n} \sum_{2d-1|k} \lambda(d) \cdot \widetilde{q}(n-k) = \sum_{k=1}^n (-1)^{n-1} \lambda(k) s_{n,k},\\
	& \sum_{k=1}^n \sum_{2d-1|k} \log(d) \cdot \widetilde{q}(n-k) = \sum_{k=1}^n (-1)^{n-1} \log(k) s_{n,k},\\
	& \sum_{k=1}^n \sum_{2d-1|k} |\mu(d)| \cdot \widetilde{q}(n-k) = \sum_{k=1}^n (-1)^{n-1} | \mu(k) | s_{n,k},\\
	& \sum_{k=1}^n \sum_{2d-1|k} J_t(d) \cdot \widetilde{q}(n-k) = \sum_{k=1}^n (-1)^{n-1} J_t(k) s_{n,k},
	\end{align*}
	where $s_{n,k}$ is the number of $(2k-1)$'s in all partitions of $n$ 
	into distinct odd parts and 
	$\widetilde{q}(n) := s_e(n) - s_o(n)$ where $s_e(n)$ and $s_o(n)$ respectively 
	denote the numbers of partitions of $n$ into even (odd) parts. 
	We can also similarly express the relations in the previous equations for any 
     special arithmetic function $a_n$ in the form of 
     \begin{align*} 
     \sum_{2d-1|n} a_d & = \sum_{k=0}^n \sum_{j=1}^k (-1)^{k-1} q(n-k) s_{k, j} a_j, 
     \end{align*} 
     where the partition function $q(n)$ denotes the number partitions of $n$ into (distinct) odd parts. 
	Moreover, since we have a direct factorization of the Lambert series generating function 
	for the sum of squares function in the form of the proposition, we may write 
	\begin{align*} 
	\sum_{k=1}^n r_2(k) \widetilde{q}(n-k) & = \sum_{k=1}^n 4 \cdot (-1)^{k+1} s_{n,k}, 
	\end{align*} 
	using the same notation as above. Similarly, we expand $r_2(n)$ as the multiple sum 
	\[
	r_2(n) = \sum_{k=0}^n \sum_{j=1}^k 4 \cdot q(n-k) (-1)^{j+1} s_{k,j}, 
	\]
	for all $n \geq 1$. 
\end{example} 

\begin{prop}
	For $|q|<1$, $0\leq \beta<\alpha$,
	\begin{align*}
	\sum_{n=1}^{\infty} a_n \frac{q^{\alpha n-\beta}}{1-q^{\alpha n-\beta}} = \frac{1}{(q^{\alpha-\beta};q^\alpha)_\infty} \sum_{n=1}^{\infty} \left( \sum_{k=1}^n (s_o(n,k)-s_e(n,k)) a_k\right)  q^n,
	\end{align*}
	where $s_o(n,k)$ and $s_e(n,k)$ denotes the number of $(\alpha k-\beta)$'s in all partition of $n$ into odd (even) number of distinct parts of the form $\alpha k-\beta$.
\end{prop}
\begin{proof}
	The proof follows from \cite[eq. 2.1]{MERCA-LSFACTTHM}, replacing $x_k$ by $q^{\alpha k-\beta}$.
\end{proof}

\begin{prop}
	For $|q|<1$, $0\leq \beta<\alpha$,
	\begin{align*}
	\sum_{n=1}^{\infty} a_n \frac{q^{\alpha n-\beta}}{1-q^{\alpha n-\beta}} = (q^{\alpha-\beta};q^\alpha)_\infty \sum_{n=1}^{\infty} \left( \sum_{k=1}^n s(n,k) a_k\right)  q^n,
	\end{align*}
	where $s(n,k)$ denotes the number of $(\alpha k-\beta)$'s in all partition of $n$ into parts of the form $\alpha k-\beta$.
\end{prop}
\begin{proof}
		We take into account the fact that
	$$\frac{q^{\alpha n-\beta}}{1-q^{\alpha n-\beta}} \cdot \frac{1}{(q^{\alpha-\beta};q^\alpha)_\infty}$$
	is the generating function for the number of $(\alpha k-\beta)$'s in all partitions of $n$ into parts of the form $\alpha k-\beta$.
	This generating function implies our result. 
\end{proof}

\subsection{The first generalized factorization theorem} 

\begin{theorem}[A General Formula for $s_{n,k}$] 
\label{theorem1_GenFormula_for_snk} 
For fixed $\alpha, \beta, \gamma, \delta \in \mathbb{Z}$ such that 
$\alpha, \gamma \geq 1$, $1 \leq \beta < \alpha$, and $1 \leq \delta < \gamma$, the 
factorization pair $(C(q), s_{n,k})$ in the generalized Lambert series factorization expanded by 
\begin{align*} 
\tag{i}
L_a(\alpha, \beta, \gamma, \delta; q) & := 
     \sum_{n \geq 1} \frac{a_n q^{\alpha n + \beta}}{1-q^{\gamma n+\delta}} = 
     \frac{1}{C(q)} \sum_{n \geq 1} \sum_{k=1}^n s_{n,k} a_k \cdot q^n, 
\end{align*} 
satisfies 
\begin{align*} 
s_{n,k} & = [q^n] \frac{q^{\alpha n + \beta}}{1-q^{\gamma n+\delta}} C(q). 
\end{align*} 
Specific interpretations of the sequence $s_{n,k}$ are given as special cases of the 
previous expansions as in the results proved in the last subsection. 
\end{theorem} 
\begin{proof} 
We begin by rewriting (i) in the form of 
\begin{align*} 
C(q) \sum_{k \geq 1} \frac{a_k q^{\alpha k+\beta}}{1-q^{\gamma k+\delta}} & = 
     \sum_{k \geq 1} \left(\sum_{n \geq 1} s_{n,k} q^n\right) a_k. 
\end{align*} 
Then if we equate the coefficients of $a_k$ in the previous equation, we see that 
\begin{align*} 
C(q) \frac{q^{\alpha k+\beta}}{1-q^{\gamma k+\delta}} & = \sum_{n \geq 1} s_{n,k} q^n, 
\end{align*} 
which implies our stated result. 
\end{proof} 

\begin{figure}[ht!]

\linespread{1}
\begin{minipage}{\linewidth} 
\begin{center} 
\tiny
\begin{equation*} 
\boxed{ 
\begin{array}{cccccccccccccccc}
 1 & 0 & 0 & 0 & 0 & 0 & 0 & 0 & 0 & 0 & 0 & 0 & 0 & 0 & 0 & 0 \\
 -1 & 1 & 0 & 0 & 0 & 0 & 0 & 0 & 0 & 0 & 0 & 0 & 0 & 0 & 0 & 0 \\
 -1 & -1 & 1 & 0 & 0 & 0 & 0 & 0 & 0 & 0 & 0 & 0 & 0 & 0 & 0 & 0 \\
 1 & -1 & -1 & 1 & 0 & 0 & 0 & 0 & 0 & 0 & 0 & 0 & 0 & 0 & 0 & 0 \\
 -1 & 0 & -1 & -1 & 1 & 0 & 0 & 0 & 0 & 0 & 0 & 0 & 0 & 0 & 0 & 0 \\
 0 & 0 & 0 & -1 & -1 & 1 & 0 & 0 & 0 & 0 & 0 & 0 & 0 & 0 & 0 & 0 \\
 1 & 2 & 0 & 0 & -1 & -1 & 1 & 0 & 0 & 0 & 0 & 0 & 0 & 0 & 0 & 0 \\
 0 & -1 & 1 & 0 & 0 & -1 & -1 & 1 & 0 & 0 & 0 & 0 & 0 & 0 & 0 & 0 \\
 0 & 0 & 0 & 1 & 0 & 0 & -1 & -1 & 1 & 0 & 0 & 0 & 0 & 0 & 0 & 0 \\
 1 & 0 & 2 & 0 & 1 & 0 & 0 & -1 & -1 & 1 & 0 & 0 & 0 & 0 & 0 & 0 \\
 0 & 0 & -1 & 1 & 0 & 1 & 0 & 0 & -1 & -1 & 1 & 0 & 0 & 0 & 0 & 0 \\
 0 & 2 & -1 & 0 & 1 & 0 & 1 & 0 & 0 & -1 & -1 & 1 & 0 & 0 & 0 & 0 \\
 0 & -1 & 0 & 1 & 0 & 1 & 0 & 1 & 0 & 0 & -1 & -1 & 1 & 0 & 0 & 0 \\
 0 & -1 & 0 & -1 & 0 & 0 & 1 & 0 & 1 & 0 & 0 & -1 & -1 & 1 & 0 & 0 \\
 0 & 0 & 0 & -1 & 0 & 0 & 0 & 1 & 0 & 1 & 0 & 0 & -1 & -1 & 1 & 0 \\
 -1 & 0 & 0 & -1 & 1 & 0 & 0 & 0 & 1 & 0 & 1 & 0 & 0 & -1 & -1 & 1 \\
\end{array}
}
\end{equation*}
\end{center} 
\subcaption*{(i) $s_{n,k}$} 
\end{minipage} 

\begin{minipage}{\linewidth} 
\begin{center} 
\tiny 
\begin{equation*} 
\boxed{ 
\begin{array}{cccccccccccccccc}
 1 & 0 & 0 & 0 & 0 & 0 & 0 & 0 & 0 & 0 & 0 & 0 & 0 & 0 & 0 & 0 \\
 1 & 1 & 0 & 0 & 0 & 0 & 0 & 0 & 0 & 0 & 0 & 0 & 0 & 0 & 0 & 0 \\
 2 & 1 & 1 & 0 & 0 & 0 & 0 & 0 & 0 & 0 & 0 & 0 & 0 & 0 & 0 & 0 \\
 2 & 2 & 1 & 1 & 0 & 0 & 0 & 0 & 0 & 0 & 0 & 0 & 0 & 0 & 0 & 0 \\
 5 & 3 & 2 & 1 & 1 & 0 & 0 & 0 & 0 & 0 & 0 & 0 & 0 & 0 & 0 & 0 \\
 7 & 5 & 3 & 2 & 1 & 1 & 0 & 0 & 0 & 0 & 0 & 0 & 0 & 0 & 0 & 0 \\
 9 & 6 & 5 & 3 & 2 & 1 & 1 & 0 & 0 & 0 & 0 & 0 & 0 & 0 & 0 & 0 \\
 15 & 11 & 7 & 5 & 3 & 2 & 1 & 1 & 0 & 0 & 0 & 0 & 0 & 0 & 0 & 0 \\
 22 & 15 & 11 & 7 & 5 & 3 & 2 & 1 & 1 & 0 & 0 & 0 & 0 & 0 & 0 & 0 \\
 27 & 21 & 14 & 11 & 7 & 5 & 3 & 2 & 1 & 1 & 0 & 0 & 0 & 0 & 0 & 0 \\
 42 & 30 & 22 & 15 & 11 & 7 & 5 & 3 & 2 & 1 & 1 & 0 & 0 & 0 & 0 & 0 \\
 55 & 41 & 30 & 22 & 15 & 11 & 7 & 5 & 3 & 2 & 1 & 1 & 0 & 0 & 0 & 0 \\
 74 & 54 & 41 & 29 & 22 & 15 & 11 & 7 & 5 & 3 & 2 & 1 & 1 & 0 & 0 & 0 \\
 101 & 77 & 56 & 42 & 30 & 22 & 15 & 11 & 7 & 5 & 3 & 2 & 1 & 1 & 0 & 0 \\
 135 & 101 & 77 & 56 & 42 & 30 & 22 & 15 & 11 & 7 & 5 & 3 & 2 & 1 & 1 & 0 \\
 170 & 132 & 99 & 76 & 55 & 42 & 30 & 22 & 15 & 11 & 7 & 5 & 3 & 2 & 1 & 1 \\
\end{array}
} 
\end{equation*} 
\end{center} 
\subcaption*{(ii) $s_{n,k}^{(-1)}$} 
\end{minipage}

\begin{minipage}{\linewidth} 
\begin{center} 
\tiny 
\begin{equation*} 
\boxed{ 
\begin{array}{cccccccccccccccc}
 1 & 0 & 0 & 0 & 0 & 0 & 0 & 0 & 0 & 0 & 0 & 0 & 0 & 0 & 0 & 0 \\
 0 & 1 & 0 & 0 & 0 & 0 & 0 & 0 & 0 & 0 & 0 & 0 & 0 & 0 & 0 & 0 \\
 0 & 1 & 1 & 0 & 0 & 0 & 0 & 0 & 0 & 0 & 0 & 0 & 0 & 0 & 0 & 0 \\
 -1 & 1 & 1 & 1 & 0 & 0 & 0 & 0 & 0 & 0 & 0 & 0 & 0 & 0 & 0 & 0 \\
 0 & 3 & 2 & 1 & 1 & 0 & 0 & 0 & 0 & 0 & 0 & 0 & 0 & 0 & 0 & 0 \\
 0 & 2 & 2 & 2 & 1 & 1 & 0 & 0 & 0 & 0 & 0 & 0 & 0 & 0 & 0 & 0 \\
 -2 & 6 & 5 & 3 & 2 & 1 & 1 & 0 & 0 & 0 & 0 & 0 & 0 & 0 & 0 & 0 \\
 1 & 7 & 6 & 4 & 3 & 2 & 1 & 1 & 0 & 0 & 0 & 0 & 0 & 0 & 0 & 0 \\
 0 & 13 & 9 & 7 & 5 & 3 & 2 & 1 & 1 & 0 & 0 & 0 & 0 & 0 & 0 & 0 \\
 -3 & 13 & 12 & 10 & 6 & 5 & 3 & 2 & 1 & 1 & 0 & 0 & 0 & 0 & 0 & 0 \\
 0 & 30 & 22 & 15 & 11 & 7 & 5 & 3 & 2 & 1 & 1 & 0 & 0 & 0 & 0 & 0 \\
 1 & 27 & 23 & 18 & 14 & 10 & 7 & 5 & 3 & 2 & 1 & 1 & 0 & 0 & 0 & 0 \\
 -3 & 54 & 41 & 29 & 22 & 15 & 11 & 7 & 5 & 3 & 2 & 1 & 1 & 0 & 0 & 0 \\
 2 & 60 & 51 & 39 & 28 & 21 & 14 & 11 & 7 & 5 & 3 & 2 & 1 & 1 & 0 & 0 \\
 0 & 90 & 68 & 54 & 40 & 30 & 22 & 15 & 11 & 7 & 5 & 3 & 2 & 1 & 1 & 0 \\
 -4 & 107 & 90 & 69 & 52 & 40 & 29 & 21 & 15 & 11 & 7 & 5 & 3 & 2 & 1 & 1 \\
\end{array}
} 
\end{equation*} 
\end{center} 
\subcaption*{(iii) $\gamma_k(n)$} 
\end{minipage}
\linespread{2}

\caption{A generalized factorization for $L_a(1, 0, 2, 1; q)$} 
\label{figure_spcase_nonsg_matrix_fact_params} 

\end{figure} 

\begin{figure}[ht!]

\linespread{1}
\begin{minipage}{\linewidth} 
\begin{center} 
\tiny
\begin{equation*} 
\boxed{ 
\begin{array}{cccccccccc}
 1 & 0 & 0 & 0 & 0 & 0 & 0 & 0 & 0 & 0 \\
 -1 & 1 & 0 & 0 & 0 & 0 & 0 & 0 & 0 & 0 \\
 -1 & -1 & 1 & 0 & 0 & 0 & 0 & 0 & 0 & 0 \\
 d & -1 & -1 & 1 & 0 & 0 & 0 & 0 & 0 & 0 \\
 -d & 0 & -1 & -1 & 1 & 0 & 0 & 0 & 0 & 0 \\
 1-d & 0 & 0 & -1 & -1 & 1 & 0 & 0 & 0 & 0 \\
 d^2 & d+1 & 0 & 0 & -1 & -1 & 1 & 0 & 0 & 0 \\
 1-d^2 & -d & 1 & 0 & 0 & -1 & -1 & 1 & 0 & 0 \\
 d-d^2 & 1-d & 0 & 1 & 0 & 0 & -1 & -1 & 1 & 0 \\
 d^3 & 0 & d+1 & 0 & 1 & 0 & 0 & -1 & -1 & 1 \\
\end{array}
}
\end{equation*}
\end{center} 
\subcaption*{(i) $s_{n,k}$} 
\end{minipage} 

\begin{minipage}{\linewidth} 
\begin{center} 
\tiny 
\begin{equation*} 
\boxed{ 
\begin{array}{cccccccccc}
 1 & 0 & 0 & 0 & 0 & 0 & 0 & 0 & 0 & 0 \\
 1 & 1 & 0 & 0 & 0 & 0 & 0 & 0 & 0 & 0 \\
 2 & 1 & 1 & 0 & 0 & 0 & 0 & 0 & 0 & 0 \\
 3-d & 2 & 1 & 1 & 0 & 0 & 0 & 0 & 0 & 0 \\
 5 & 3 & 2 & 1 & 1 & 0 & 0 & 0 & 0 & 0 \\
 7 & 5 & 3 & 2 & 1 & 1 & 0 & 0 & 0 & 0 \\
 -d^2-d+11 & 7-d & 5 & 3 & 2 & 1 & 1 & 0 & 0 & 0 \\
 15 & 11 & 7 & 5 & 3 & 2 & 1 & 1 & 0 & 0 \\
 22 & 15 & 11 & 7 & 5 & 3 & 2 & 1 & 1 & 0 \\
 -d^3-2 d+30 & 22-d & 15-d & 11 & 7 & 5 & 3 & 2 & 1 & 1 \\
\end{array}
} 
\end{equation*} 
\end{center} 
\subcaption*{(ii) $s_{n,k}^{(-1)}$} 
\end{minipage}

\begin{minipage}{\linewidth} 
\begin{center} 
\tiny 
\begin{equation*} 
\boxed{ 
\begin{array}{cccccccccc}
 1 & 0 & 0 & 0 & 0 & 0 & 0 & 0 & 0 & 0 \\
 0 & 1 & 0 & 0 & 0 & 0 & 0 & 0 & 0 & 0 \\
 0 & 1 & 1 & 0 & 0 & 0 & 0 & 0 & 0 & 0 \\
 -d & 1 & 1 & 1 & 0 & 0 & 0 & 0 & 0 & 0 \\
 0 & 3 & 2 & 1 & 1 & 0 & 0 & 0 & 0 & 0 \\
 0 & 2 & 2 & 2 & 1 & 1 & 0 & 0 & 0 & 0 \\
 -d^2-d & 7-d & 5 & 3 & 2 & 1 & 1 & 0 & 0 & 0 \\
 d & 7 & 6 & 4 & 3 & 2 & 1 & 1 & 0 & 0 \\
 0 & 13 & 9 & 7 & 5 & 3 & 2 & 1 & 1 & 0 \\
 -d^3-2 d & 14-d & 13-d & 10 & 6 & 5 & 3 & 2 & 1 & 1 \\
\end{array}
} 
\end{equation*} 
\end{center} 
\subcaption*{(iii) $\gamma_k(n)$} 
\end{minipage}
\linespread{2}

\caption{A generalized factorization for $L_a(1, 0, 2, 1; d, q)$} 
\label{figure_spcase_nonsg_matrix_fact_paramsd_v2} 

\end{figure} 

\begin{cor} 
\label{cor_ConsequenceOfThm1_ConnectionBetweenOrdFactThms} 
Let $\alpha \geq 1$ and $0 \leq \beta < \alpha$ be integers and suppose that $\delta \in \mathbb{Z}$. 
Suppose that 
\begin{align*} 
\sum_{n \geq 1} \frac{a_n q^{n}}{1-q^{n}} = 
     \frac{1}{C(q)} \sum_{n \geq 0} \sum_{k=1}^n s_{n,k} a_k \cdot q^n. 
\end{align*} 
and that 
\begin{align*} 
\sum_{n \geq 1} \frac{a_n q^{\alpha n-\beta+\delta}}{1-q^{\alpha n-\beta}} = 
     \frac{1}{C(q)} \sum_{n \geq 0} \sum_{k=1}^n s_{n,k}(\alpha, \beta; \delta) a_k \cdot q^n. 
\end{align*} 
Then we have that 
\begin{align*} 
s_{n,k}(\alpha, \beta; \delta) & = s_{n-\delta,\alpha k-\beta}. 
\end{align*} 
In particular, when $C(q) \equiv (q; q)_{\infty}$ we have that 
the result in the equation above holds 
for $s_{n,k} = s_o(n, k) - s_e(n, k)$ where $s_o(n, k)$ and $s_e(n, k)$ 
are respectively the number of $k$'s in all 
partitions of $n$ into an odd (even) number of distinct parts. 
\end{cor} 
\begin{proof} 
The proof follows from the second statement in the theorem which 
provides a generating function for $s_{n,k}$ for all $n, k \geq 1$. 
\end{proof} 

\begin{remark}[Conjectures on Degenerate Cases of the Theorem] 
\label{remark_conj_degen_cases_of_thm1} 
One distinction to be made between the factorization result in the theorem and the 
results in \cite{MERCA-SCHMIDT1,MERCA-SCHMIDT2} is that for $\alpha = \gamma$ and 
$\beta = \delta$ as in the generalized expansion of 
\eqref{eqn_GenLambertSeries_LaAlphaBetaq_def} 
from the introduction, the square matrix, $A_n := (s_{i,j})_{1 \leq i,j \leq n}$, tends to 
\emph{not} be invertible for $\alpha > 1$. 
The statement in Theorem \ref{theorem1_GenFormula_for_snk} does however allow us to formulate 
generalized analogs to the results in the references for some special cases of the 
parameters in (i) of the theorem. 
For example, if we consider the factorizations of the Lambert series expansions of the form 
\begin{align*} 
L_a(1, 0, 2, 1; q) = \sum_{n \geq 1} \frac{a_n q^n}{1-q^{2n+1}}, 
\end{align*} 
with $C(q) := (q; q)_{\infty}$, 
we obtain non-singular matrices, $A_n$, as we have come to expect in the results from 
\cite{MERCA-SCHMIDT1,MERCA-SCHMIDT2} whose properties are summarized by 
Figure \ref{figure_spcase_nonsg_matrix_fact_params} where we define the corresponding 
inverse matrix entries to be of the form 
\begin{align*} 
s_{n,k}^{(-1)} & := \sum_{d|n} p(d-k) \cdot \gamma_k(n/d), 
\end{align*} 
for some fixed sequence of arithmetic functions $\gamma_k$. 
We notice by experimentation with the integer sequences database \cite{OEIS} that for $k > \gamma+\delta$ 
in the expansion of $L(1, 0, \gamma, \delta; q)$ we appear to have that 
$\gamma_k(n)$ is related to partition functions with special generating functions of some 
underspecified sort. Experimentally, we also conjecture that\footnote{ 
     \underline{\emph{Notation}}: 
     \emph{Iverson's convention} compactly specifies 
     boolean-valued conditions and is equivalent to the 
     \emph{Kronecker delta function}, $\delta_{i,j}$, as 
     $\Iverson{n = k} \equiv \delta_{n,k}$. 
     Similarly, $\Iverson{\mathtt{cond = True}} \equiv 
                 \delta_{\mathtt{cond}, \mathtt{True}}$ 
     in the remainder of the article. 
}
\begin{align} 
\label{eqn_DegenCasesOfThm_example_v1}
s_{n,k}^{(-1)} & = p(n-k) - \sum_{i=1}^n p\left(\frac{n-i}{2i+1}-k\right) \Iverson{n \equiv i \bmod 2i+1} \\ 
\notag 
     & \phantom{=p(n-k)\ } + 
     \sum_{m=2}^n \sum_{i=1}^n p\left(\frac{n-p(m+1)i-p(m-1)}{p(m+1)(2i+1)} - k\right) \times \\ 
\notag
     & \phantom{=p(n-k)+\sum\sum\ \ } \times 
     \Iverson{n \equiv p(m+1)i+p(m-1) \bmod p(m+1)(2i+1)}. 
\end{align} 
We can similarly define the analogous related 
expansions for the class of generalized Lambert series of the form 
\begin{align*} 
L_a(1, 0, 2, 1; d, q) = \sum_{n \geq 1} \frac{a_n q^n}{1-d \cdot q^{2n+1}}, 
\end{align*} 
so that the corresponding matrices of $s_{n,k}$ are also invertible. 
Figure \ref{figure_spcase_nonsg_matrix_fact_paramsd_v2} summarizes the characteristic 
expansions of the factorization when $C(q) := (q; q)_{\infty}$. 
From our computational experimentation with the 
form of this series, we see that each of the entries in the 
tables listed in the figure are polynomials in $d$. 
Moreover, as in the first degenerate series case, we conjecture a related result that 
\begin{align*} 
s_{n,k}^{(-1)} & = p(n-k) - \sum_{i=1}^n d^i \cdot p\left(\frac{n-i}{2i+1}-k\right) \Iverson{n \equiv i \bmod 2i+1} \\ 
     & \phantom{=p(n-k)\ } + 
     \sum_{m=2}^n \sum_{i=1}^n d^i \cdot p\left(\frac{n-p(m+1)i-p(m-1)}{p(m+1)(2i+1)} - k\right) \times \\ 
     & \phantom{=p(n-k)+\sum\sum\ \ } \times 
     \Iverson{n \equiv p(m+1)i+p(m-1) \bmod p(m+1)(2i+1)} \\ 
     & \phantom{=p(n-k)\ } + 
     p_{n,k}(d), 
\end{align*} 
where each $p_{n,k}(d)$ is a small polynomial in $d$ and is only non-zero in the rows $n$ indexed by the 
sequence $\{13, 22, 31, 37, 40, 49, 52, 58, 62, 67, 73, \ldots\}$. 

We can generalized the result for the first degenerate case in \eqref{eqn_DegenCasesOfThm_example_v1} 
somewhat to state that if 
\begin{align*} 
L_a(1, 0, \alpha, 1; q) = \sum_{n \geq 1} \frac{a_n q^n}{1-q^{\alpha n+1}} = 
     \frac{1}{C(q)} \sum_{n \geq 0} \sum_{k=1}^n s_{n,k}(\alpha, 1) a_k \cdot q^n, 
\end{align*} 
for some $\alpha \geq 2$, then we have that 
\begin{align} 
\notag 
s_{n,k}^{(-1)}(\alpha, 1) & := p(n-k) - \sum_{i=1}^n p\left(\frac{n-i}{\alpha i+1}-k\right) 
     \Iverson{n \equiv i \bmod \alpha i+1} + p_{n,k}(\alpha, 1), 
\end{align} 
where the sequence of $p_{n,k}(\alpha, 1)$ is integer-valued, mostly zero, and for the 
rows $n$ where $p_{n,k}(\alpha, 1)$ is non-zero this function is only non-zero for small cases 
of the columns $k$ with $k \ll n$. 
For example, when $\alpha := 3$ we have that the functions $p_{n,k}(3, 1)$ are non-zero only for the 
small column cases in the rows $n \in \{21,37,53,65,69,85,93,101,117,121,133,149, \ldots\}$, 
when $\alpha := 4$ the functions $p_{n,k}(4, 1)$ are non-zero only for the rows 
$n \in \{31,56,81,101,106,131,146, \ldots\}$, and 
when $\alpha := 5$ the functions $p_{n,k}(5, 1)$ are non-zero only for the rows 
$n \in \{43,79,115,145, \ldots\}$. The in-order non-zero indexed rows appear to have the same 
values for all $\alpha \geq 3$, i.e., $p_{21,k}(3, 1) = p_{31,k}(4, 1) = p_{43,k}(5, 1)$, and so on. 
\end{remark} 

\section{Variants of the Generalized Factorization Theorems} 
\label{Section_Variants_and_Examples}

The results in this section avoid the typically singular behavior of the matrices 
corresponding to the factorizations of \eqref{eqn_GenLambertSeries_LaAlphaBetaq_def} by construction and 
redefining the sequence $\bar{a}_n$ implicit to the more general factorization statement in 
\eqref{eqn_GenFactThmExp_def_v1}. 
The next theorem provides the generalized analog to the variant of the ordinary factorization theorems 
proved in \cite[\S 3]{MERCA-SCHMIDT2}. 

\begin{theorem}[Another Generalized Factorization Theorem] 
\label{theorem2_AnotherGenFactThm}
If we define the factorization pair $(C(q), s_{n,k})$ in \eqref{eqn_GenFactThmExp_def_v1} 
implicitly through \eqref{eqn_snk_snkinv_det-based_recrels} by 
\begin{align*}
s_{n,k}^{(-1)} & := \sum_{d|n} [q^{d-k}] \frac{1}{C(q)} \cdot \gamma(n/d), 
\end{align*} 
for some fixed arithmetic functions $\gamma(n)$ and $\widetilde{\gamma}(n) := \sum_{d|n} \gamma(d)$, 
then we have that the sequence of $\bar{a}_n$ is given by the following formula for all $n \geq 1$: 
\begin{align*} 
\bar{a}_n & = 
     \sum_{\substack{d|n \\ d \equiv \beta \bmod \alpha}} a_{\frac{d-\beta}{\alpha}} \widetilde{\gamma}(n/d). 
\end{align*} 
\end{theorem} 
\begin{proof} 
The proof of this result is similar to the proof given in \cite[Thm.\ 3.3]{MERCA-SCHMIDT2}. 
In particular, we begin by noticing that 
\begin{align*} 
\bar{a}_n & = \sum_{k=1}^n s_{n,k}^{(-1)} [q^k] \left(\sum_{d=1}^k 
     \frac{a_d q^{\alpha d+\beta}}{1-q^{\alpha d+\beta}}\right). 
\end{align*} 
Then if we let $c_n := [q^n] 1 / C(q)$ and let 
\begin{align*} 
t_{k,d}(\alpha, \beta) & := [q^k] \frac{q^{\alpha d+\beta}}{1-q^{\alpha d+\beta}} C(q), 
\end{align*} 
where $t_{i,d}(\alpha, \beta) = 0$ whenever $i < \alpha d+\beta$, 
we have that for each $1 \leq d \leq n$ 
\begin{align*}
[a_d] \bar{a}_n & = \sum_{k=1}^n s_{n,k}^{(-1)} t_{k,d}(\alpha, \beta) \\ 
     & = 
     \sum_{k=d}^n \left(\sum_{r|n} c_{r-k} \gamma(n/r)\right) t_{k,d}(\alpha, \beta) \\ 
     & = 
     \sum_{r|n} \left(\sum_{i=d}^r c_{r-i} t_{i,d}(\alpha, \beta)\right) \gamma(n/r) \\ 
     & = 
     \sum_{r|n} \left(\sum_{i=0}^r c_{r-i} t_{i,d}(\alpha, \beta)\right) \gamma(n/r), 
\end{align*} 
where the inner sum is generated by 
\begin{align*} 
[q^r] \frac{1}{\cancel{C(q)}} \frac{q^{\alpha d+\beta}}{1-q^{\alpha d+\beta}} \cancel{C(q)} & = 
     \Iverson{\alpha d+\beta | r}. 
\end{align*} 
Thus we have that 
\begin{align*} 
[a_d] \bar{a}_n & = \sum_{\substack{r|n \\ \alpha d + \beta | r}} \gamma(n/d), 
\end{align*} 
and then that 
\begin{align*} 
\bar{a}_n & = \sum_{\substack{d|n \\ d \equiv \beta \bmod \alpha}} a_{\frac{d-\beta}{\alpha}} 
     \sum_{\substack{r|n \\ d | r}} \gamma(n/r) \\ 
     & = 
     \sum_{\substack{d|n \\ d \equiv \beta \bmod \alpha}} a_{\frac{d-\beta}{\alpha}} 
     \sum_{r | \frac{n}{d}} \gamma\left(\frac{n}{dr}\right). 
     \qedhere
\end{align*} 
\end{proof} 

We choose to phrase the next few results cited in 
Example \ref{example_AppsOfTheVariantThm} in terms of the special case function 
$C(q) \equiv (q; q)_{\infty}$ whose reciprocal generates Euler's partition function as 
$p(n) := [q^n] (q; q)_{\infty}^{-1}$ in the 
first form of the factorization assumed by Theorem \ref{theorem2_AnotherGenFactThm}. 
We note that these identities could just as well be expanded in terms of other 
reciprocal generating functions for special partition functions are defined through the 
choice of $C(q)$, such as the partition function $q(n)$ corresponding to 
$C(q) \equiv (-q; q)_{\infty}^{-1}$. 

\begin{example}[Applications of the Theorem] 
\label{example_AppsOfTheVariantThm}
We begin by noticing that for a fixed function $C(q)$, 
the general expansions of the factorization result 
stated in \eqref{eqn_GenFactThmExp_def_v1} implies that we have 
\begin{align*} 
\bar{a}_n & = \sum_{k=1}^n s_{n,k}^{(-1)} B_{k-1}, 
\end{align*} 
where $B_k$ depends implicitly only on the selection of $a_n$ according to the 
next sums over the paired function $b_m$ in 
\eqref{eqn_GenLambertSeries_LaAlphaBetaq_def} for $k \geq 1$: 
\begin{align*} 
B_{k-1} & = \sum_{k=1}^n s_{n,k} a_k \\ 
     & = 
     b_k + \sum_{k=1}^n [q^n] \frac{1}{C(q)} \cdot b_{n-k} \\ 
     & = 
     \sum_{\alpha d-\beta|n} a_d + \sum_{k=1}^n [q^n] \frac{1}{C(q)} \cdot 
     \left(\sum_{\alpha d-\beta|n-k} a_d\right). 
\end{align*} 
For example, when we take $(a_n, \gamma(n)) := (1, n^t)$ for some 
$t \in \mathbb{C}$ in Theorem \ref{theorem2_AnotherGenFactThm}, we obtain the 
result 
\begin{align*} 
 & \sum_{\substack{d|n \\ d \equiv \beta \bmod \alpha}} \sigma_{t}\left(\frac{n}{d}\right) \\ 
     & = 
     \sum_{k=1}^n \left(\sum_{d|n} p(d-k) (n/d)^{t}\right) \left[ 
     b_k(1; \alpha, \beta) + \sum_{s = \pm 1} 
     \sum_{j=1}^{\left\lfloor \frac{\sqrt{24k+1}-s}{6} \right\rfloor} 
     b_{k-\frac{j(3j+s)}{2}}(1; \alpha, \beta)\right], 
\end{align*} 
where we define the notation for $b_k(1; \alpha, \beta)$ to be the special case of the 
restricted divisor sums over $a_n$ given by 
\[ 
b_k(a_n; \alpha, \beta) := \sum_{\alpha d-\beta|k} a_d. 
\]
If we next define $(a_n, \gamma(n)) := (n^s, n^t), (n^s, \phi(n))$, respectively, 
we similarly obtain the following identities as corollaries of the theorem above: 
\begin{align*} 
 & \sum_{\substack{d|n \\ d \equiv \beta \bmod \alpha}} \left(\frac{d-\beta}{\alpha}\right)^{s} \cdot \sigma_{t}(n/d) \\ 
     & = 
     \sum_{k=1}^n \left(\sum_{d|n} p(d-k) (n/d)^{t}\right) \left[ 
     b_k(n^s; \alpha, \beta) + \sum_{s = \pm 1} 
     \sum_{j=1}^{\left\lfloor \frac{\sqrt{24k+1}-s}{6} \right\rfloor} 
     b_{k-\frac{j(3j+s)}{2}}(n^s; \alpha, \beta)\right] \\ 
 & n \times \sum_{\substack{d|n \\ d \equiv \beta \bmod \alpha}} \frac{1}{d} \cdot 
     \left(\frac{d-\beta}{\alpha}\right)^{t} \\ 
     & = 
     \sum_{k=1}^n \left(\sum_{d|n} p(d-k) \phi(n/d)\right) \left[ 
     b_k(n^t; \alpha, \beta) + \sum_{s = \pm 1} \sum_{j=1}^{\left\lfloor \frac{\sqrt{24k+1}-s}{6} \right\rfloor} 
     b_{k-\frac{j(3j+s)}{2}}(n^t; \alpha, \beta)\right].
\end{align*} 
The generalized expansions of the well-known classical Lambert series stated in 
\eqref{eqn_WellKnown_LamberSeries_Examples} also imply the next corollaries 
of the theorem for some fixed $s, t \in \mathbb{C}$. 
\begin{align*} 
 & \sum_{\substack{d|n \\ d \equiv \beta \bmod \alpha}} \log(d) \\ 
 & \quad = \sum_{k=1}^n \left(\sum_{d|n} p(d-k) \Lambda(n/d)\right) \left[ 
     b_k(1; \alpha, \beta) + \sum_{s = \pm 1} 
     \sum_{j=1}^{\left\lfloor \frac{\sqrt{24k+1}-s}{6} \right\rfloor} 
     b_{k-\frac{j(3j+s)}{2}}(1; \alpha, \beta)\right] \\ 
 & \sum_{\substack{d|n \\ d \equiv \beta \bmod \alpha}} \left(\frac{d-\beta}{\alpha}\right)^{s} \cdot 
     \left(\frac{n}{d}\right)^t \\ 
 & = \sum_{k=1}^n \left(\sum_{d|n} p(d-k) J_t(n/d)\right) \left[ 
     b_k(n^s; \alpha, \beta) + \sum_{s = \pm 1} \sum_{j=1}^{\left\lfloor \frac{\sqrt{24k+1}-s}{6} \right\rfloor} 
     b_{k-\frac{j(3j+s)}{2}}(n^s; \alpha, \beta))\right]. 
\end{align*} 
We also have expansions of related identities involving the sum of squares function, 
$r_2(n) = [q^n] \vartheta_3(q)^2$, in the special case where 
$(\alpha, \beta, a_n) := (2, 1, 4 \cdot (-1)^{n+1})$. 
In particular, for any prescribed arithmetic functions $a_n$ and $\gamma(n)$ such that 
$\widetilde{\gamma}(n) := \sum_{d|n} \gamma(d)$, we have  expansions involving 
$r_2(n)$ given by 
\begin{align*} 
\sum_{\substack{d|n \\ d\text{ odd}}} a_{\frac{d-1}{2}} r_2(n/d) & = 
     \sum_{k=1}^n \sum_{d|n} 4 \cdot (-1)^{n/d+1} p(d-k) \Biggl[b_k(a_n; 2, 1) \\ 
     & \phantom{=\sum\sum 4 \cdot\ }  + 
     \sum_{s = \pm 1} \sum_{j=1}^{\left\lfloor \frac{\sqrt{24k+1}-s}{6} \right\rfloor} 
     b_{k-\frac{j(3j+s)}{2}}(a_n; 2, 1)\Biggr] 
\end{align*} 
and expanded by 
\begin{align*} 
\sum_{\substack{d|n \\ d\text{ odd}}} & 4 \cdot (-1)^{(d+1)/2} \widetilde{\gamma}(n/d) \\ 
     & = 
     \sum_{k=1}^n \sum_{d|n} p(d-k) \gamma(n/d) \cdot [q^k] (q; q)_{\infty} \vartheta_3(q)^2 \\ 
\sum_{\substack{d|n \\ d\text{ odd}}} & (-1)^{(d+1)/2} \left(r_2\left(\frac{n}{d}\right) 
     -4 \cdot d\left(\frac{n}{2d}\right) \Iverson{\frac{n}{d}\text{ even}}\right) \\ 
     & = 
     \sum_{k=1}^n \sum_{d|n} p(d-k) (-1)^{n/d+1} \cdot [q^k] (q; q)_{\infty} \vartheta_3(q)^2, 
\end{align*} 
where $d(n) \equiv \sigma_0(n)$ denotes the divisor function and where the 
powers of the Jacobi theta function, $\vartheta_3(q) := 1 + 2 \sum_{n \geq 1} q^{n^2}$, 
generate the more general sums of the sums of $k$ squares functions, 
$r_k(n) := [q^n] \vartheta_3(q)^k$. 
\end{example} 

\section{Factorization theorems for Lambert series over convolutions of arithmetic functions} 
\label{Section_LSFactThms_CvlOfTwoFns_and_Apps}

\subsection{Overview and definitions} 

Given two prescribed arithmetic functions $f$ and $g$ we define their \emph{convolution}, 
or \emph{Dirichlet convolution}, denoted by $h = f \ast g$, to be the function 
\begin{align*} 
(f \ast g)(n) & := \sum_{d|n} f(d) g(n/d), 
\end{align*} 
for all natural numbers $n \geq 1$ \cite[\S 2.6]{APOSTOL-ANUMT}. The usual M\"obius inversion 
result is stated in terms of convolutions as follows where 
$\mu$ is the M\"obius function: $h = f \ast 1$ if and only if $f = h \ast \mu$. 
There is a natural connection between the coefficients of the Lambert series of an arithmetic function $a_n$ and its corresponding \emph{Dirichlet generating function}, 
$\DGF(a_n; s) := \sum_{n \geq 1} a_n / n^s$. 
Namely, we have that for any $s \in \mathbb{C}$ such that $\Re(s) > 1$ 
$$b_n = [q^n] \sum_{n \geq 1} \frac{a_n q^n}{1-q^n} \quad\text{ if and only if }\quad 
  \DGF(b_n; s) = \DGF(a_n; s) \zeta(s), $$ where $\zeta(s)$ is the Riemann zeta function. Moreover, we can further connect the coefficients of the Lambert series over a convolution of arithmetic functions to its 
associated Dirichlet series by noting that $\DGF(f \ast g; s) = \DGF(f; s) \cdot \DGF(g; s)$. 

In this section, we consider the generalized Lambert series factorization theorems in the context of 
\eqref{eqn_GenLambertSeries_LaAlphaBetaq_def} 
where the function $a_n$ is defined to be a convolution of two arithmetic functions. 
Since we have not yet explored factorization theorems of this type in the references 
(\cf \cite{MERCA-SCHMIDT1,MERCA-SCHMIDT2}), we first state a 
few factorization results for Lambert series 
over these convolution functions in the ``\emph{ordinary}'' case of 
\eqref{eqn_GenLambertSeries_LaAlphaBetaq_def} where $(\alpha, \beta) := (1, 0)$. 
More precisely, we give the statements and proofs of 
Proposition \ref{prop_OnePossibleLSFact} and a closely-related theorem 
next and then proceed to evaluate several 
consequences and special case formulations following from these results below. 

\subsection{Main results} 

\begin{table}[ht!]

\linespread{1}
\begin{center} 
\tiny
\begin{equation*} 
\boxed{ 
\begin{array}{ccccccccccccccccccccc}
 1 & 0 & 0 & 0 & 0 & 0 & 0 & 0 & 0 & 0 & 0 & 0 & 0 & 0 & 0 & 0 & 0 & 0 & 0 & 0 & 0 \\
 1 & 1 & 0 & 0 & 0 & 0 & 0 & 0 & 0 & 0 & 0 & 0 & 0 & 0 & 0 & 0 & 0 & 0 & 0 & 0 & 0 \\
 1 & 1 & 0 & 0 & 0 & 0 & 0 & 0 & 0 & 0 & 0 & 0 & 0 & 0 & 0 & 0 & 0 & 0 & 0 & 0 & 0 \\
 1 & 2 & 1 & 0 & 0 & 0 & 0 & 0 & 0 & 0 & 0 & 0 & 0 & 0 & 0 & 0 & 0 & 0 & 0 & 0 & 0 \\
 1 & 1 & 0 & 0 & 0 & 0 & 0 & 0 & 0 & 0 & 0 & 0 & 0 & 0 & 0 & 0 & 0 & 0 & 0 & 0 & 0 \\
 1 & 3 & 2 & 0 & 0 & 0 & 0 & 0 & 0 & 0 & 0 & 0 & 0 & 0 & 0 & 0 & 0 & 0 & 0 & 0 & 0 \\
 1 & 1 & 0 & 0 & 0 & 0 & 0 & 0 & 0 & 0 & 0 & 0 & 0 & 0 & 0 & 0 & 0 & 0 & 0 & 0 & 0 \\
 1 & 3 & 3 & 1 & 0 & 0 & 0 & 0 & 0 & 0 & 0 & 0 & 0 & 0 & 0 & 0 & 0 & 0 & 0 & 0 & 0 \\
 1 & 2 & 1 & 0 & 0 & 0 & 0 & 0 & 0 & 0 & 0 & 0 & 0 & 0 & 0 & 0 & 0 & 0 & 0 & 0 & 0 \\
 1 & 3 & 2 & 0 & 0 & 0 & 0 & 0 & 0 & 0 & 0 & 0 & 0 & 0 & 0 & 0 & 0 & 0 & 0 & 0 & 0 \\
 1 & 1 & 0 & 0 & 0 & 0 & 0 & 0 & 0 & 0 & 0 & 0 & 0 & 0 & 0 & 0 & 0 & 0 & 0 & 0 & 0 \\
 1 & 5 & 7 & 3 & 0 & 0 & 0 & 0 & 0 & 0 & 0 & 0 & 0 & 0 & 0 & 0 & 0 & 0 & 0 & 0 & 0 \\
 1 & 1 & 0 & 0 & 0 & 0 & 0 & 0 & 0 & 0 & 0 & 0 & 0 & 0 & 0 & 0 & 0 & 0 & 0 & 0 & 0 \\
 1 & 3 & 2 & 0 & 0 & 0 & 0 & 0 & 0 & 0 & 0 & 0 & 0 & 0 & 0 & 0 & 0 & 0 & 0 & 0 & 0 \\
 1 & 3 & 2 & 0 & 0 & 0 & 0 & 0 & 0 & 0 & 0 & 0 & 0 & 0 & 0 & 0 & 0 & 0 & 0 & 0 & 0 \\
 1 & 4 & 6 & 4 & 1 & 0 & 0 & 0 & 0 & 0 & 0 & 0 & 0 & 0 & 0 & 0 & 0 & 0 & 0 & 0 & 0 \\
 1 & 1 & 0 & 0 & 0 & 0 & 0 & 0 & 0 & 0 & 0 & 0 & 0 & 0 & 0 & 0 & 0 & 0 & 0 & 0 & 0 \\
 1 & 5 & 7 & 3 & 0 & 0 & 0 & 0 & 0 & 0 & 0 & 0 & 0 & 0 & 0 & 0 & 0 & 0 & 0 & 0 & 0 \\
 1 & 1 & 0 & 0 & 0 & 0 & 0 & 0 & 0 & 0 & 0 & 0 & 0 & 0 & 0 & 0 & 0 & 0 & 0 & 0 & 0 \\
 1 & 5 & 7 & 3 & 0 & 0 & 0 & 0 & 0 & 0 & 0 & 0 & 0 & 0 & 0 & 0 & 0 & 0 & 0 & 0 & 0 \\
 1 & 3 & 2 & 0 & 0 & 0 & 0 & 0 & 0 & 0 & 0 & 0 & 0 & 0 & 0 & 0 & 0 & 0 & 0 & 0 & 0 \\
 1 & 3 & 2 & 0 & 0 & 0 & 0 & 0 & 0 & 0 & 0 & 0 & 0 & 0 & 0 & 0 & 0 & 0 & 0 & 0 & 0 \\
 1 & 1 & 0 & 0 & 0 & 0 & 0 & 0 & 0 & 0 & 0 & 0 & 0 & 0 & 0 & 0 & 0 & 0 & 0 & 0 & 0 \\
 1 & 7 & 15 & 13 & 4 & 0 & 0 & 0 & 0 & 0 & 0 & 0 & 0 & 0 & 0 & 0 & 0 & 0 & 0 & 0 & 0 \\
 1 & 2 & 1 & 0 & 0 & 0 & 0 & 0 & 0 & 0 & 0 & 0 & 0 & 0 & 0 & 0 & 0 & 0 & 0 & 0 & 0 \\
 1 & 3 & 2 & 0 & 0 & 0 & 0 & 0 & 0 & 0 & 0 & 0 & 0 & 0 & 0 & 0 & 0 & 0 & 0 & 0 & 0 \\
 1 & 3 & 3 & 1 & 0 & 0 & 0 & 0 & 0 & 0 & 0 & 0 & 0 & 0 & 0 & 0 & 0 & 0 & 0 & 0 & 0 \\
 1 & 5 & 7 & 3 & 0 & 0 & 0 & 0 & 0 & 0 & 0 & 0 & 0 & 0 & 0 & 0 & 0 & 0 & 0 & 0 & 0 \\
 1 & 1 & 0 & 0 & 0 & 0 & 0 & 0 & 0 & 0 & 0 & 0 & 0 & 0 & 0 & 0 & 0 & 0 & 0 & 0 & 0 \\
 1 & 7 & 12 & 6 & 0 & 0 & 0 & 0 & 0 & 0 & 0 & 0 & 0 & 0 & 0 & 0 & 0 & 0 & 0 & 0 & 0 \\
 1 & 1 & 0 & 0 & 0 & 0 & 0 & 0 & 0 & 0 & 0 & 0 & 0 & 0 & 0 & 0 & 0 & 0 & 0 & 0 & 0 \\
 1 & 5 & 10 & 10 & 5 & 1 & 0 & 0 & 0 & 0 & 0 & 0 & 0 & 0 & 0 & 0 & 0 & 0 & 0 & 0 & 0 \\
 1 & 3 & 2 & 0 & 0 & 0 & 0 & 0 & 0 & 0 & 0 & 0 & 0 & 0 & 0 & 0 & 0 & 0 & 0 & 0 & 0 \\
 1 & 3 & 2 & 0 & 0 & 0 & 0 & 0 & 0 & 0 & 0 & 0 & 0 & 0 & 0 & 0 & 0 & 0 & 0 & 0 & 0 \\
 1 & 3 & 2 & 0 & 0 & 0 & 0 & 0 & 0 & 0 & 0 & 0 & 0 & 0 & 0 & 0 & 0 & 0 & 0 & 0 & 0 \\
 1 & 8 & 19 & 18 & 6 & 0 & 0 & 0 & 0 & 0 & 0 & 0 & 0 & 0 & 0 & 0 & 0 & 0 & 0 & 0 & 0 \\
 1 & 1 & 0 & 0 & 0 & 0 & 0 & 0 & 0 & 0 & 0 & 0 & 0 & 0 & 0 & 0 & 0 & 0 & 0 & 0 & 0 \\
 1 & 3 & 2 & 0 & 0 & 0 & 0 & 0 & 0 & 0 & 0 & 0 & 0 & 0 & 0 & 0 & 0 & 0 & 0 & 0 & 0 \\
 1 & 3 & 2 & 0 & 0 & 0 & 0 & 0 & 0 & 0 & 0 & 0 & 0 & 0 & 0 & 0 & 0 & 0 & 0 & 0 & 0 \\
 1 & 7 & 15 & 13 & 4 & 0 & 0 & 0 & 0 & 0 & 0 & 0 & 0 & 0 & 0 & 0 & 0 & 0 & 0 & 0 & 0 \\
 1 & 1 & 0 & 0 & 0 & 0 & 0 & 0 & 0 & 0 & 0 & 0 & 0 & 0 & 0 & 0 & 0 & 0 & 0 & 0 & 0 \\
 1 & 7 & 12 & 6 & 0 & 0 & 0 & 0 & 0 & 0 & 0 & 0 & 0 & 0 & 0 & 0 & 0 & 0 & 0 & 0 & 0 \\
 1 & 1 & 0 & 0 & 0 & 0 & 0 & 0 & 0 & 0 & 0 & 0 & 0 & 0 & 0 & 0 & 0 & 0 & 0 & 0 & 0 \\
 1 & 5 & 7 & 3 & 0 & 0 & 0 & 0 & 0 & 0 & 0 & 0 & 0 & 0 & 0 & 0 & 0 & 0 & 0 & 0 & 0 \\
 1 & 5 & 7 & 3 & 0 & 0 & 0 & 0 & 0 & 0 & 0 & 0 & 0 & 0 & 0 & 0 & 0 & 0 & 0 & 0 & 0 \\
 1 & 3 & 2 & 0 & 0 & 0 & 0 & 0 & 0 & 0 & 0 & 0 & 0 & 0 & 0 & 0 & 0 & 0 & 0 & 0 & 0 \\
 1 & 1 & 0 & 0 & 0 & 0 & 0 & 0 & 0 & 0 & 0 & 0 & 0 & 0 & 0 & 0 & 0 & 0 & 0 & 0 & 0 \\
 1 & 9 & 26 & 34 & 21 & 5 & 0 & 0 & 0 & 0 & 0 & 0 & 0 & 0 & 0 & 0 & 0 & 0 & 0 & 0 & 0 \\
 1 & 2 & 1 & 0 & 0 & 0 & 0 & 0 & 0 & 0 & 0 & 0 & 0 & 0 & 0 & 0 & 0 & 0 & 0 & 0 & 0 \\
 1 & 5 & 7 & 3 & 0 & 0 & 0 & 0 & 0 & 0 & 0 & 0 & 0 & 0 & 0 & 0 & 0 & 0 & 0 & 0 & 0 \\
\end{array}
}
\end{equation*}
\end{center} 
\linespread{2}

\caption{The functions $\ds_{j,1}(n)$ for columns $1 \leq j \leq 21$ and $1 \leq n \leq 50$} 
\label{table_dsjgn_gEquivOne_listings} 

\end{table} 

\begin{prop}[One Possible Factorization] 
\label{prop_OnePossibleLSFact} 
Let $f$ and $g$ denote non-identically-zero arithmetic functions. 
Suppose that we have an ordinary Lambert series factorization for any prescribed 
arithmetic function $a_n$ of the form 
\begin{align*} 
\tag{i}
\sum_{n \geq 1} \frac{a_n q^n}{1-q^n} & = \frac{1}{C(q)} \sum_{n \geq 1} \sum_{k=1}^n 
     s_{n,k} a_k \cdot q^n, 
\end{align*} 
so that, for example, when $C(q) := (q; q)_{\infty}$ we have that $s_{n,k} = s_o(n, k) - s_e(n, k)$ 
where $s_o(n, k)$ and $s_e(n, k)$ respectively denote the number of $k$'s in all 
partitions of $n$ into an odd (even) number of distinct parts. 
Then we have one possible formulation of a factorization theorem for the Lambert series over the 
convolution function $h = f \ast g$ expanded as 
\begin{align*} 
\tag{ii} 
\sum_{n \geq 1} \frac{(f \ast g)(n) q^n}{1-q^n} & = \frac{1}{C(q)} \sum_{n \geq 1} \sum_{k=1}^n 
     \widetilde{s}_{n,k}(g) f(k) \cdot q^n, 
\end{align*} 
where 
\begin{equation*} 
\tag{iii}
\widetilde{s}_{n,k}(g) = \sum_{j=1}^n s_{n,kj} \cdot g(j). 
\end{equation*} 
\end{prop} 
\begin{proof} 
It is apparent by the expansions on the left-hand-side of (ii) that there is some sequence of 
$\widetilde{s}_{n,k}(g)$ depending on the function $g$ that satisfies the factorization of the 
form in (i) when $a_n \mapsto (f \ast g)(n)$. 
For a fixed $k \geq 1$, we begin by evaluating the coefficients of $f(k)$ on the 
right-hand-side of (ii) as follows: 
\begin{align*} 
[f(k)] \left(\sum_{n \geq 1} \sum_{i=1}^n s_{n,i} (f \ast g)(i) \cdot q^n\right) 
     & = \sum_{n \geq 1} \sum_{i=1}^n [f(k)] s_{n,i} \sum_{d|i} f(d) g(i/d) \cdot q^n \\ 
     & = \sum_{n \geq 1} \sum_{i=1}^n s_{n,i} g(i/k) \Iverson{k|i} \cdot q^n \\ 
     & = \sum_{n \geq 1} \left(\sum_{j=1}^n s_{n,kj} \cdot g(j)\right) q^n. 
\end{align*} 
Thus we have that the formula for $\widetilde{s}_{n,k}(g)$ given in (iii) is correct 
as claimed. 
\end{proof} 

By Corollary \ref{cor_ConsequenceOfThm1_ConnectionBetweenOrdFactThms} 
we can easily generalize this result to the analogous series 
expansions of \eqref{eqn_GenLambertSeries_LaAlphaBetaq_def}. 
However, in these cases the matrix of $s_{n,k}$ is typically 
singular so that we are unable to formulate an analog to the following theorem for the 
inverse sequences which implies many useful and interesting new results discussed in the 
following examples. 

\noindent 
\textbf{Notation}: 
For a fixed arithmetic function $g$ with $g(1) := 1$, 
let the functions $\ds_{j,g}(n)$ be defined recursively for natural numbers 
$j \geq 1$ as 
\begin{align*} 
\ds_{j,g}(n) & := 
     \begin{cases} 
     g_{\pm}(n), & \text{ if $j = 1$; } \\ 
     \sum\limits_{\substack{d|n \\ d>1}} g(d) \ds_{j-1,g}(n/d), & \text{ if $j > 1$, } 
     \end{cases} 
\end{align*} 
where $g_{\pm}(n) := g(n) \Iverson{n > 1} - \delta_{n,1}$ and 
let the notation for the $k$-shifted partition function be defined as 
$p_k(n) := p(n-k)$ for $k \geq 1$. If we let the function $\widetilde{\ds}_{j,g}(n)$ 
denote the $j$-fold convolution of $g$ with itself, i.e., that 
$$\widetilde{\ds}_{j,g}(n) = \underset{\text{$n$ times}}{\underbrace{\left(g_{\pm} \ast g \ast \cdots \ast g\right)}}(n), $$ 
then we can prove easily by induction that for all $m, n \geq 1$ we have the expansion 
$$\ds_{m,g}(n) = \sum_{i=0}^{m-1} \binom{m-1}{i} (-1)^{m-1-i} \cdot \widetilde{\ds}_{i+1,g}(n). $$ 
We then define the following notation for the sums of the variant convolution functions for 
use in the theorem below when $n \geq 1$: 
\begin{align*} 
D_{n,g}(n) & := \sum_{j=1}^n \ds_{2j,g}(n) = \sum_{m=1}^{\lfloor n/2 \rfloor} \sum_{i=0}^{2m-1} 
     \binom{2m-1}{i} (-1)^{i+1} \widetilde{\ds}_{i+1,g}(n). 
\end{align*} 
A listing of the values of the functions $\ds_{j,g}(n)$ in the special form when 
$g(n) \equiv 1$ is tabulated in Table \ref{table_dsjgn_gEquivOne_listings}. The sum 
\[
\sum_{i=1}^n \ds_{2i,1}(n) \mapsto \{0, 1,1,2,1,3,1,4,2,3,1,8,1,3,3,8,\ldots\}, 
\] 
and its M\"obius transform appear in the integer sequences database as the number of 
perfect partitions of $n$ \cite[A002033, A174726]{OEIS}. 
The next result provides the form of the inverse sequences, $s_{n,k}^{(-1)}$, 
in the expansions proved in the proposition immediately above. 

\begin{table}[ht!]
\linespread{1}
\begin{center} 
\small
\begin{equation*} 
\boxed{ 
\begin{array}{cccccccccc}
 1 & 0 & 0 & 0 & 0 & 0 & 0 & 0 & 0 & 0 \\
 0 & 1 & 0 & 0 & 0 & 0 & 0 & 0 & 0 & 0 \\
 0 & -1 & 1 & 0 & 0 & 0 & 0 & 0 & 0 & 0 \\
 0 & 0 & -1 & 1 & 0 & 0 & 0 & 0 & 0 & 0 \\
 0 & -1 & -1 & -1 & 1 & 0 & 0 & 0 & 0 & 0 \\
 0 & 1 & 1 & -1 & -1 & 1 & 0 & 0 & 0 & 0 \\
 0 & -1 & -1 & 0 & -1 & -1 & 1 & 0 & 0 & 0 \\
 0 & 1 & 0 & 1 & 0 & -1 & -1 & 1 & 0 & 0 \\
 0 & -1 & 2 & 0 & 0 & 0 & -1 & -1 & 1 & 0 \\
 0 & 2 & -1 & -1 & 2 & 0 & 0 & -1 & -1 & 1 \\
 0 & -2 & -1 & 1 & -1 & 1 & 0 & 0 & -1 & -1 \\
 0 & 2 & 3 & 2 & 0 & 1 & 1 & 0 & 0 & -1 \\
 0 & -2 & -2 & -1 & 0 & 0 & 0 & 1 & 0 & 0 \\
 0 & 3 & -1 & -2 & 0 & -1 & 2 & 0 & 1 & 0 \\
 0 & -3 & 4 & 1 & 3 & 0 & -1 & 1 & 0 & 1 \\
 0 & 3 & -3 & 2 & -2 & 0 & -1 & 1 & 1 & 0 \\
 0 & -4 & -2 & -1 & -2 & 1 & 0 & -1 & 0 & 1 \\
 0 & 5 & 5 & -3 & 0 & 1 & 0 & -1 & 1 & 0 \\
 0 & -5 & -4 & 1 & 0 & -1 & 0 & 0 & -1 & 0 \\
 0 & 5 & -2 & 4 & 4 & -2 & 0 & -1 & -1 & 1 \\
 0 & -6 & 8 & -2 & -3 & -1 & 3 & 1 & -1 & -1 \\
\end{array}
}
\end{equation*}
\end{center} 
\linespread{2}

\caption{The sequences $\rho_{n,1}^{(i)}$ for the first rows $1 \leq n \leq 21$ and 
         columns $1 \leq i \leq 10$}. 
\label{table_rhonki_keq1_seqs} 

\end{table} 

\begin{theorem}[Inverse Sequences]
\label{claim_snk_inverses}
When $C(q) := (q; q)_{\infty}$, the inverse sequences from 
Proposition \ref{prop_OnePossibleLSFact} satisfy 
\begin{align*} 
\sum_{d|n} s_{n,k}^{(-1)}(g) & = p_k(n) + \sum_{j=1}^{\Omega(n)} (p_k \ast \ds_{2j,g})(n) \\ 
     & = p_k(n) + (p_k \ast D_{n,g})(n), 
\end{align*} 
which by M\"obius inversion is equivalent to 
\begin{align*} 
s_{n,k}^{(-1)}(g) & = (p_k \ast \mu)(n) + \sum_{j=1}^{\Omega(n)} (p_k \ast \ds_{2j,g} \ast \mu)(n) \\ 
     & = (p_k \ast \mu)(n) + (p_k \ast D_{n,g} \ast \mu)(n), 
\end{align*} 
where $D_{n,g}(n)$ is defined as in the remarks given above. 
We note that strictly speaking replacing the 
upper bound on the first sum above with $n$ in place of $\Omega(n)$ has no effect on the value of the sum. 
\end{theorem} 
\begin{proof}[Proof of the Claim]  
Let the proposed inverse sequence function be defined in the following notation which we denote in 
shorthand by $\widehat{s}_k \equiv \widehat{s}_k(n)$: 
\begin{align*} 
\widehat{s}_{n,k}^{(-1)}(g) & := (p_k \ast \mu)(n) + (p_k \ast D_{n,g} \ast \mu)(n). 
\end{align*} 
We begin as in the proof of Theorem 3.2 in \cite{MERCA-SCHMIDT1} and consider the ordinary, 
non-convolved Lambert series over the function $\widehat{s}_k$. 
More precisely, by the expansion in (i) of the proposition we must show that 
\begin{align*} 
\sum_{d|n} \widehat{s}_{d,k}^{(-1)} & := p_k(n) + (p_k \ast D_{n,g})(n) \\ 
     & = \sum_{m=0}^n \sum_{j=1}^{n-m} \widetilde{s}_{n-m,j}(g) \cdot \widehat{s}_{j,k}^{(-1)} \cdot p(m). 
\end{align*} 
For $n, k, i \geq 1$ with $k, i \leq n$, let the coefficient functions, $\rho_{n,k}^{(i)}$ 
be defined as 
\[
\rho_{n,k}^{(i)} := \sum_{j=1}^n s_{n,ij} \cdot \widetilde{s}_{j,k}^{(-1)}. 
\] 
Then for any fixed arithmetic function $h$ we can prove (by considering the related expansions of the 
factorizations in (ii) of the proposition for $\widehat{s}_k \ast g$) that 
\begin{align*} 
\tag{i}
t_{n,k}(h) & := \sum_{j=1}^n s_{n,j} \cdot (\widetilde{s}_{n,k}^{(-1)} \ast h)(j) 
      = \sum_{i=1}^n \rho_{n,k}^{(i)} \cdot h(i). 
\end{align*} 
It remains to show that 
\[
\tag{ii} 
\sum_{m=0}^n t_{n-m,k}(h) \cdot p(m) = (p_k \ast h)(n). 
\] 
Since we can expand the left-hand-side of the previous sum as 
\begin{align*} 
\sum_{m=0}^n \sum_{i=1}^{n-m} \rho_{n-m,k}^{(i)} \cdot h(i) \cdot p(m) & = 
     \sum_{i=1}^n h(i) \underset{ := u_{n,k}^{(i)}}{\underbrace{\left( 
     \sum_{m=0}^n \rho_{n-m,k}^{(i)} \cdot p(m)\right)}}, 
\end{align*} 
to complete the proof of (ii) we need to prove a subclaim 
that (I) $u_{n,k}^{(i)} = 0$ if $i \not{\mid} n$; and (II) 
if $i | n$ then $u_{n,k}^{(i)} = p\left(\frac{n}{i}-k\right)$. 

\noindent
\textit{Proof of Subclaim}: 
For $i := 1$, this is clearly the case since $\rho_{n-m,k}^{(i)} = \Iverson{n-m=k}$. 
For subsequent cases of $i \geq 2$, it is apparent that 
$$\rho_{n,k}^{(i)} = \rho_{n-(k-1)i,1}$$ much as in the cases of the tables for the inverse sequences, 
$s_{n,k}^{(-1)} = (p_k \ast \mu)(n)$, given in the reference \cite[\S 3]{MERCA-SCHMIDT1} 
(see Table \ref{table_rhonki_keq1_seqs}). Finally, we claim that generating functions for the 
sequences of $u_{n,k}^{(i)}$ for each $i \geq 2$ are expanded in the form of
\[
\sum_{n \geq 0} u_{n,k}^{(i)} \cdot q^n = 
     \prod_{j=1}^{i-1} (q^j; q^i)_{\infty} \times \frac{q^{ik}}{(q; q)_{\infty}} = 
     q^{ik} \cdot \sum_{n \geq 0} p(n) q^{in}, 
\] 
which we see by comparing coefficients on the right-hand-side of the 
previous equation implies our claim. 

\noindent
\textit{Completing the Proof of the Inverse Formula}: 
What we have shown by proving (ii) above is an inverse formula for an ordinary 
Lambert series factorization over the sequence of 
$a_j := (\widetilde{s}_{n,k}^{(-1)} \ast g)(j)$. 
In particular, by M\"obius inversion (ii) shows that we have 
\begin{align*} 
(f \ast g)(n) = (s_{n,k}^{(-1)} \ast g)(n) & \iff 
(f - s_{n,k}^{(-1)}) \ast g \equiv 0 \\ 
     & \implies 
     f_n = s_{n,k}^{(-1)},\ \text{ when $g \not{\equiv} 0$. } 
\end{align*} 
More to the point, when we define $f_n := \widetilde{s}_{n,k}^{(-1)}(g)$ 
where by convenience and experimental suggestion we let 
$$\widetilde{s}_{n,k}^{(-1)}(g) = s_{n,k}^{(-1)} \ast t_{n,k}^{(-1)}(g),$$ 
for some convolution-wise factorization of this inverse sequence, we can now prove the 
exact formula for the inverse sequence claimed in the theorem statement. 
In the forward direction, we suppose that 
\[t_{n,k}^{(-1)}(g) = D_{n,g}(n) + \varepsilon(n),\] 
where $\varepsilon(n) = \delta_{n,1}$ 
is the multiplicative identity and then see from the formulas for $D_{n,g}(n)$ discussed 
before the claim that $g \ast (D_{n,g}+\varepsilon) = \varepsilon$, which proves that 
our inverse formula is correct in this case. 
Conversely, if we require that 
\[
s_{n,k}^{(-1)} \ast t_{n,k}^{(-1)}(g) \ast g = s_{n,k}^{(-1)} 
\] 
for all $n$ and choices of the function $g$, we must have that 
$t_{n,k}^{(-1)} \ast g = \varepsilon$, and so we see that 
$t_{n,k}^{(-1)} = D_{n,g} + \varepsilon$ as required. 
That is to say, we have proved our result using the implicit statement that 
$t \ast g = \varepsilon$ if and only if $t = D_{n,g} + \varepsilon$, i.e., 
that $t = D_{n,g} + \varepsilon$ is the unique function such that 
$t \ast g = \varepsilon$ for all $n \geq 1$, 
a result which we do not prove here and only mention for the sake of brevity. 
\end{proof} 

\subsection{Corollaries and applications} 

We have several immediate consequences of the theorem, which in some respects follows naturally as 
a corollary of the expansions of the proposition above. 
Some of these applications are discussed in the next example. 

\begin{example}[Applications] 
To ease the notation in the applications that follow, for integers $n \geq 0$ 
we define the functions $B_n(f \ast g)$  in the shorthand of 
\begin{align*} 
B_n(f \ast g) & := [q^{n+1}]\left(\sum_{m \geq 1} \frac{(f \ast g)(n) q^n \times (q; q)_{\infty}}{1-q^n}\right) \\ 
     & \phantom{:} = 
     b_{n+1}(f \ast g) + \sum_{s = \pm 1} \sum_{j=1}^{\left\lfloor \frac{\sqrt{24k+25}-s}{6} \right\rfloor} 
     (-1)^j b_{n+1-\frac{j(3j+s)}{2}}(f \ast g), 
\end{align*} 
where $b_k(f \ast g) = \sum_{d|n} (f \ast g)(d)$ is as on the left-hand-side of 
\eqref{eqn_GenLambertSeries_LaAlphaBetaq_def} for $(\alpha, \beta) := (1, 0)$. 
Then using the notation in Proposition \ref{prop_OnePossibleLSFact} and in the claim immediately above, 
we see that our new expansions imply that 
\begin{align*} 
f(n) & = \sum_{k=1}^n s_{n,k}^{(-1)}(g) \cdot B_{k-1}(f, g). 
\end{align*} 
We will employ several well-known Dirichlet convolution results which provide applications of the 
previous expansions. 
For example, given any fixed $t \in \mathbb{C}$ we have the known identity that 
$\Id_t = \sigma_t \ast \mu$ where $\Id_t(n) = n^t$ is the $t^{th}$ power function. 
In particular, this identity implies each of the following identities: 
\begin{align*} 
\mu(n) & = \sum_{k=1}^n s_{n,k}^{(-1)}(\sigma_t) \cdot B_{k-1}(\sigma_t) \\ 
\sigma_t(n) & = \sum_{k=1}^n s_{n,k}^{(-1)}(\mu) \cdot B_{k-1}(\sigma_t). 
\end{align*} 
Similarly, since $\Id_1 = \phi \ast 1$, 
$\sigma_1 = \phi \ast \Id_1$, and $\Lambda = \log \ast \mu$, we have the 
following related expansions: 
\begin{align*} 
\phi(n) & = \sum_{k=1}^n s_{n,k}^{(-1)}(\Id_0) \cdot B_{k-1}(\sigma_1) \\ 
\tag{$\dagger$} 
1 & = \sum_{k=1}^n s_{n,k}^{(-1)}(\phi) \cdot B_{k-1}(\sigma_1) \\ 
\phi(n) & = \sum_{k=1}^n s_{n,k}^{(-1)}(\Id_1) \cdot B_{k-1}(\sigma_1 \ast 1) \\ 
n & = \sum_{k=1}^n s_{n,k}^{(-1)}(\phi) \cdot B_{k-1}(\sigma_1 \ast 1) \\ 
\log(n) & = \sum_{k=1}^n s_{n,k}^{(-1)}(\mu) \cdot B_{k-1}(\log) \\ 
\mu(n) & = \sum_{k=1}^n s_{n,k}^{(-1)}(\log) \cdot B_{k-1}(\log). 
\end{align*} 
We emphasize that the formulas listed above are special in nature due to the 
unconventional dependence of the inverse sequences, $s_{n,k}^{(-1)}(g)$ on 
multiple $j$-fold convolutions of the function $g$ for $1 \leq j \leq n$. 
The next explicit computation illustrates this property for the 
identity tagged in ($\dagger$) above when $n := 4$: 
\begin{align*} 
1 & = \left[(\phi_{\pm} \ast \phi \ast \phi \ast \phi)(4) - 3 \cdot (\phi_{\pm} \ast \phi \ast \phi)(4) + 
     4 \cdot (\phi_{\pm} \ast \phi)(4) - 2 \cdot \phi(4) + 2\right] \sigma(1) \\ 
     & \phantom{=\ } - 
     \left[(\phi_{\pm} \ast \phi)(2) - \phi(2) + 1\right](\sigma(1) - \sigma(2)) - 
     (\sigma(1) + \sigma(2) - \sigma(3)) \\ 
     & \phantom{=\ } - 
     (\sigma(2) + \sigma(3) - \sigma(4)). 
\end{align*} 
Many other examples of known convolution results to which we can apply our new 
expansions are found in the references \cite{APOSTOL-ANUMT,SURVEY-DSERIES}. 
\end{example}

\begin{cor}[Formulas for the Dirichlet Inverse] 
For any prescribed arithmetic function $f$ defined such that $f(1) = 1$, we have 
a formula for its Dirichlet inverse function given by 
\begin{align*}
f^{-1}(n) & = \sum_{k=1}^n \left((p_k \ast \mu)(n) + (p_k \ast D_{n,f} \ast \mu)(n)\right) \cdot 
     [q^{k-1}] \frac{(q; q)_{\infty}}{1-q}. 
\end{align*} 
\end{cor} 
\begin{proof} 
The proof follows from Theorem \ref{claim_snk_inverses} applied in the form of 
Proposition \ref{prop_OnePossibleLSFact}. In particular, since 
$f^{-1} \ast f = \delta_{n,1}$ by definition, the right-hand-side of our Lambert series expansion 
over the convolved function $a_n := f^{-1} \ast f$ is given by $q / (1-q)$. 
\end{proof} 

We compare this formula for the Dirichlet inverse of an arithmetic function to the other primary 
known recursive divisor sum formula defining the inverse function of $f$ 
given by \cite[\S 2.7]{APOSTOL-ANUMT} 
\begin{align*} 
f^{-1}(1) = \frac{1}{f(1)},\ f^{-1}(n) = -\frac{1}{f(1)} \sum_{\substack{d|n \\ d<n}} 
     f(n/d) f^{-1}(d)\ \text{ for $n > 1$, } 
\end{align*} 
for $f(1) \neq 0$ and where it is known that $(f \ast g)^{-1} = f^{-1} \ast g^{-1}$ if 
$f(1), g(1) \neq 0$. The Dirichlet inverse of a function $f$ exists precisely when 
$f(1) \neq 0$, and by scaling our result in the corollary matches these cases as well. 

We also note that given any prescribed sequence of $b(n)$ we can generate $b(n)$ by the 
Lambert series over $b \ast \mu$. This implies that we have recurrence relations for any 
arithmetic function $b$ defined such that $b(n) = 0$ for all $n < 0$ expanded in the 
following two forms where $s_{n,k} := [q^n] (q; q)_{\infty} q^k / (1-q^k)$: 
\begin{align*} 
b(n) & = \sum_{k=1}^n \left(p_k \ast \mu+p_k \ast D_{n,\mu} \ast \mu\right)(n) \left[ 
     b(k) + \sum_{s = \pm 1} \sum_{j=1}^k (-1)^j b\left(k-\frac{j(3j+s)}{2}\right)\right] \\ 
b(n) & = \sum_{j=1}^n \sum_{k=1}^j \left(\sum_{i=1}^{\left\lfloor j/k \right\rfloor} 
     s_{j,ki} \cdot \mu(i)\right) b(k) p(n-j). 
\end{align*} 

\begin{cor}[Convolution Formulas for Arithmetic Functions] 
Suppose that we have two prescribed arithmetic functions $f$ and $h$ and we seek the form 
of a third $g$ satisfying $f \ast g = h \ast \mu$ for all $n \geq 1$. Then we have a formula for the 
function $g$ expanded in the form of 
\begin{align*}
g(n) & = \sum_{k=1}^n \left((p_k \ast \mu)(n) + (p_k \ast D_{n,f} \ast \mu)(n)\right) \times \\ 
     & \phantom{=\sum\ } \times 
     \left(h(k) + \sum_{s = \pm 1} \sum_{j=1}^{\left\lfloor \frac{\sqrt{24k+1}-s}{6} \right\rfloor} 
     (-1)^j h\left(k-\frac{j(3j+s)}{2}\right)\right). 
\end{align*} 
\end{cor} 
\begin{proof} 
This result is an immediate consequence of Proposition \ref{prop_OnePossibleLSFact} and the formula 
for the inverse sequences defined by Theorem \ref{claim_snk_inverses}. 
\end{proof} 

\section{Conclusions} 

\subsection{Summary} 

In Section \ref{Section_GenFactThms_and_Examples} and 
Section \ref{Section_Variants_and_Examples} we proved two key new variations of the 
Lambert series factorization theorems studied in the references 
\cite{MERCA-SCHMIDT1,MERCA-SCHMIDT2,MERCA-LSFACTTHM,SCHMIDT-LSFACTTHM} each of 
which effectively generalize the expansions of these initial model results to the more 
general cases of the Lambert series expansions defined by 
\eqref{eqn_GenLambertSeries_LaAlphaBetaq_def} in the introduction. 
The propositions for special cases and the examples discussed in 
Section \ref{subSection_GenFactThmV1_SpCases_and_Examples} motivate the 
formulation of Theorem \ref{theorem1_GenFormula_for_snk} 
which provides generating functions for the 
factorization parameters, $s_{n,k}$, in the most general cases and which 
motivate its immediate consequence stated in 
Corollary \ref{cor_ConsequenceOfThm1_ConnectionBetweenOrdFactThms} 
providing the connection between our new generalized factorization 
theorems and the known ordinary cases established by the references. 

We have similarly proved a variation of the generalized factorization 
theorem stated by Theorem \ref{theorem2_AnotherGenFactThm} in 
Section \ref{Section_Variants_and_Examples} along the lines 
of the same motivations for the corresponding ``ordinary'' case result in 
\cite{MERCA-SCHMIDT2}. 
That is, instead of starting with the right-hand-side factorizations in 
\eqref{eqn_GenFactThmExp_def_v1}, we define the form of 
$s_{n,k}^{(-1)}$ which determines the entries of the 
matrix of $s_{n,k}$ and then proceed to explore the form of the resulting factorizations 
based on our definition. It turns out that the formulas in this case provide a natural 
analog to the corresponding result for the Lambert series cases expanded in the 
reference. 

The truly new results unique to this article are stated and proved in 
Section \ref{Section_LSFactThms_CvlOfTwoFns_and_Apps} where we provide a variant of the 
Lambert series factorization theorems given in the previous sections and in the references 
for a convolution of two arithmetic functions. These new results lead to applications for 
several well-known Dirichlet convolutions of classical arithmetic functions which are 
expanded by multiple $j$-fold convolutions of one of these classical functions with itself. 

\subsection{A note on further generalizations} 

It is possible to consider even more general factorization theorem results for 
Lambert series expansions of the form 
\begin{align*} 
L_a(\alpha, \beta; c, d; q) := \sum_{n \geq 1} 
     \frac{a_n c^n q^{\alpha n-\beta}}{1-d \cdot q^{\alpha n-\beta}}, 
\end{align*} 
for $\alpha, \beta$ defined in the conventions stated in the introduction and for 
some fixed constants $c, d \in \mathbb{C}$ defined such that the series in the 
previous equation converges when $|q^{\alpha}| < 1$, though for the most part 
we do not consider such expansions here. We do however motivate these expansions 
when $c := 1$ in the second degenerate case of Theorem \ref{theorem1_GenFormula_for_snk} 
discussed in Remark \ref{remark_conj_degen_cases_of_thm1} by demonstrating the 
similarities and relations to the corresponding ordinary degenerate case of this series 
where $d \equiv \pm 1$. 

Except in a few rare cases of Lambert series expansions related to 
special functions where $d := -1$, the consideration of the additional parameters 
$c,d$ does not seem to add much utility to phrasing our notable new identities for 
the classical special functions. We also note that in this particular case we have a 
transformation identity for the ordinary cases of the Lambert series expansions in 
\eqref{eqn_GenLambertSeries_LaAlphaBetaq_def} of the form \cite{MERCA-SCHMIDT2} 
\begin{align*}
\sum_{n=1}^{\infty} \frac{a_n q^n}{1+q^n} 
&  = \sum_{n=1}^{\infty} \frac{a_n q^n}{1-q^n}  - 2\sum_{n=1}^{\infty} \frac{a_n q^{2n}}{1-q^{2n}} \\
&  = \sum_{n=1}^{\infty} \frac{b_n q^n}{1-q^n} ,
\end{align*}
where 
$$
b_n = \begin{cases}
a_n, & \text{for $n$ odd,}\\
a_n-2a_{n/2} & \text{for $n$ even.}\\
\end{cases}
$$
We therefore leave such generalizations as a topic suggested for 
future work separate from the results we have established within this article. 

\subsection{Other Variants of the Factorization Theorems}
\label{subSection_Concl_OtherFactThmExps} 

There are many other variations of the Lambert series factorization theorems we have proved in 
Proposition \ref{prop_OnePossibleLSFact} and in the previous sections whose 
properties we have still left yet unexplored. 
As an example of what other useful factorizations are possible, 
for any fixed arithmetic function $a_n$ let $A(x) := \sum_{n \leq x} a_n$. Then we may 
consider the properties, i.e., inverses, etc., of the factorizations of the 
following forms, among several other possibilities: 
\begin{align*} 
\sum_{n \geq 1} \frac{A(n) q^n}{1-q^n} & = \frac{1}{(q; q)_{\infty}} 
     \sum_{n \geq 1} \sum_{k=1}^n s_{1,n,k} \cdot a_k \cdot q^n \\ 
\sum_{n \geq 1} \frac{a_n q^n}{1-q^n} & = \frac{1}{(q; q)_{\infty}} 
     \sum_{n \geq 1} \sum_{k=1}^n s_{2,n,k} \cdot A(k) \cdot q^n. 
\end{align*} 
The factorization expanded in the second of the previous equations has the benefit that to analyze the 
average order $A(x)$ of the prescribed arithmetic function it is only necessary to examine the 
properties of the sequence of inverses, $s_{2,n,k}^{(-1)}$, and the form of the function $a_n \ast 1$. 
Then using the method of proof in Theorem \ref{theorem1_GenFormula_for_snk}, we can show that 
\begin{align*} 
s_{1,n,k} & = d(n) + \sum_{s = \pm 1} \sum_{j=1}^{\left\lfloor \frac{\sqrt{24k+1}-s}{6} \right\rfloor} 
     (-1)^j d\left(n-\frac{j(3j+s)}{2}\right) - \sum_{i=1}^{k-1} s_{n,i}, 
\end{align*} 
and that 
\begin{align*} 
s_{n,k} & = \sum_{i=0}^{n-k} s_{2,n,n-i} \quad\text{ and }\quad 
     s_{2,n,k} = s_{n,k}-s_{n,k+1}, 
\end{align*} 
where $s_{n,k} = [q^n] q^k / (1-q^k) \cdot (q; q)_{\infty}$ denotes the difference of 
the number of $k$'s in all partitions of $n$ into an odd number of distinct parts and into 
an even number of parts. These identities also imply that 
\begin{align*} 
(A \ast 1)(n) = \sum_{k=1}^n s_{1,n,k} \cdot a_k \quad\text{ and }\quad 
(a \ast 1)(n) = \sum_{k=1}^n s_{2,n,k} \cdot A(k), 
\end{align*} 
for any prescribed arithmetic function $a_n$. 
Additionally, we can show that 
\begin{align*} 
s_{1,n,k}^{(-1)} & = s_{n,k}-s_{n-1,k} \Iverson{n>1} \\ 
     & = \sum_{d|n} p(d-k) \mu(n/d) - 
     \sum_{d|n-1} p(d-k) \mu((n-1)/d) \Iverson{n > 1} \\ 
s_{2,n,k}^{(-1)} & = \sum_{j=1}^n s_{j,k} 
     = \sum_{j=1}^n \sum_{d|j} p(d-k) \mu(j/d), 
\end{align*} 
which then implies new exact identities such as the following expansions involving 
$a_n$ for $n \geq 2$: 
\begin{align*} 
a_n & = \sum_{k=1}^n \sum_{s \in \{0, 1\}} \sum_{d|n-s} (-1)^s p(d-k) \mu\left(\frac{n-s}{d}\right) \cdot 
     [q^k] \left(\sum_{m \geq 1} (A \ast 1)(m) q^m \times (q; q)_{\infty}\right) \\ 
A(n) & = \sum_{k=1}^n \sum_{j=1}^n \sum_{d|j} p(d-k) \mu(j/d) \cdot 
     [q^k] \left(\sum_{m \geq 1} (a \ast 1)(m) q^m \times (q; q)_{\infty}\right). 
\end{align*} 
Some properties of the 
generalized cases of these factorization theorems are apparent by inspection of the examples cited above. 
In particular, suppose that we have factorizations of the form 
\begin{align*} 
\sum_{n \geq 1} \frac{a_n q^n}{1-q^n} & = \frac{1}{C(q)} \sum_{n \geq 1} \sum_{k=1}^n 
     s_{n,k} \cdot a_k \cdot q^n \\ 
\sum_{n \geq 1} \frac{a_n q^n}{1-q^n} & = \frac{1}{C(q)} \sum_{n \geq 1} \sum_{k=1}^n 
     \widetilde{s}_{n,k} \left(\sum_{i=1}^k b_i a_i\right) \cdot q^n, 
\end{align*} 
for some sequence $b_i$ of non-zero functions. Then it is not difficult to prove that 
\begin{align*} 
\widetilde{s}_{n,k} & = \frac{s_{n,k}}{b_k} - \frac{s_{n,k+1}}{b_{k+1}} 
     \quad\text{ and }\quad 
\widetilde{s}_{n,k}^{(-1)} = \sum_{i=1}^n b_i \cdot s_{i,k}^{(-1)}. 
\end{align*} 
To give some additional possibilities for factorization theorems to consider based on the 
results in this article and in the references, we note that we may expand 
\begin{align*} 
\sum_{n \geq 1} \frac{a_n q^n}{1-q^n} & = \frac{1}{C(q)} \sum_{n \geq 1} \sum_{k=1}^n 
     s_{n,k}(a) \cdot a_k^{-1} \cdot q^n \\ 
\sum_{n \geq 1} \frac{A(n) q^n}{1-q^n} & = \frac{1}{C(q)} \sum_{n \geq 1} \sum_{k=1}^n 
     s_{n,k}(a) \left(\sum_{i=1}^k \frac{a_i}{i}\right) \cdot q^n
\end{align*} 
where $a_n^{-1}$ is the Dirichlet inverse of $a_n$, and that we may expand 
\begin{align*} 
\sum_{n \geq 1} \frac{\widetilde{a}_n q^n}{1-q^n} & = \frac{1}{C(q)} \sum_{n \geq 1} \sum_{k=1}^n 
     s_{n,k}(\gamma) \cdot A(k) \cdot q^n
\end{align*} 
where we define $s_{n,k}(\gamma)$ implicitly by 
$s_{n,k}^{(-1)} := \sum_{d|n} [q^{d-k}] \frac{1}{C(q)} \cdot \gamma(n/d)$ and where we 
conjecture that the function 
$\widetilde{a}_n$ is defined explicitly as the following sum involving the $A(k)$ and the $j$-fold 
convolution functions $D_{n,\gamma}(k)$ from 
Section \ref{Section_LSFactThms_CvlOfTwoFns_and_Apps}: 
\begin{align*} 
\widetilde{a_n} & = \sum_{d|n} \sum_{\substack{r|d \\ r > 1}} A(n/d) D_{n,\gamma}(r) \mu(d/r) + 
     \sum_{d|n} A(d) \mu(n/d). 
\end{align*} 
For the most part, except for the remarks given above in this subsection, 
we leave these and other expansions of analogous factorization theorems as a topic of future research 
based on the work in this article and in the references 
\cite{MERCA-SCHMIDT1,MERCA-SCHMIDT2,MERCA-LSFACTTHM,SCHMIDT-LSFACTTHM}. 

\subsection*{Acknowledgments} 

The authors thank the referees for their helpful insights and comments on 
preparing the manuscript.

\end{document}